\documentclass[leqno,a4paper]{article}
 
\usepackage{amsmath,amssymb,amsthm}
\usepackage{graphicx,color}
\usepackage{eucal}

\newtheorem{theorem}{Theorem}[section]
\newtheorem{lemma}[theorem]{Lemma}
\newtheorem{proposition}[theorem]{Proposition}

\theoremstyle{remark}
\newtheorem{remark}[theorem]{Remark}
\theoremstyle{definition}

\newtheorem{notation}[theorem]{Notation}
\newtheorem{example}[theorem]{Example}

\numberwithin{equation}{section}

\newcounter{counter_a}

\newcounter{num}
\renewcommand{\thenum}{\roman{num}}

\newcommand\la{\lambda}

\newcommand\de{\delta}

\newcommand\cA{{\mathcal A}}
\newcommand\cD{{\mathcal D}}
\newcommand\cH{{\mathcal H}}

\newcommand\tpi{{\pi}}
\newcommand\trho{{\rho}}
\newcommand\tka{{\kappa}}

\newcommand\dinf{d_\infty}
\newcommand\linf{_\infty}
\newcommand\blinf{_{\!\!\infty}}
\newcommand\Dis{{\mathfrak D}}

\renewcommand{\matrix}[2]{\left( \!\! \begin{array}{#1} #2 \end{array} \!\! \right)}

\newcommand{\ov}[1]{\overline{#1}}

\DeclareMathOperator\Real{Re}
\DeclareMathOperator\Imag{Im}
\renewcommand\Re{\Real}
\renewcommand\Im{\Imag}

\DeclareMathOperator\spn{span}
\renewcommand\d{{\rm d}}
\newcommand{\e}{{\rm e}}

\newcommand{\ii}{{\rm i}}

\newcommand{\N}{{\mathbb N}}

\newcommand{\RR}{{\mathbb R}}
\newcommand{\CC}{{\mathbb C}}
\newcommand{\Z}{{\mathbb Z}}

\newcommand\eps{\varepsilon}

\newcommand\wt{\widetilde}

\newcounter{marke}
\newcommand{\bl}{\begin{list}{\roman{marke})}{\usecounter{marke}
\topsep 0 cm \itemsep 0cm}}
\newcommand{\el}{\end{list}}

\newcommand{\defeq}{\mathrel{\mathop:}=}
\newcommand{\defequ}{\mathrel{\mathop:}\hspace*{-0.72ex}&=}

\newcommand\ds{\displaystyle}

\newcommand\D{\partial}

\newcommand\dt{{\displaystyle{\frac{\d}{\d t}}}}
\newcommand\dx{{\displaystyle{\frac{\d}{\d x}}}}
\newcommand\dr{{\displaystyle{\frac{\d}{\d r}}}}

\newcommand\spt{\sigma_{\rm p}}
\newcommand\sess{\sigma_{\rm ess}}
\newcommand\sessreg{\sigma_{\rm ess}^{\rm \,r}}
\newcommand\sesssing{\sigma_{\rm ess}^{\rm \,s}}

\DeclareMathOperator\sign{sign}
\DeclareMathOperator\ran{ran}

\newcommand\vect[2]{\binom{#1}{#2}}

\newcommand{\address}[1]{\noindent{\small #1}\par }
\newcommand{\email}[1]{\noindent{\small\tt #1}\par\medskip }

\newcommand\void[1]{}

\newcommand{\tp}{\widetilde p}
\newcommand{\tq}{\widetilde q}
\newcommand{\tb}{\widetilde b}
\newcommand{\tc}{\widetilde c}
\newcommand{\td}{\widetilde d}

\newcommand{\tA}{\widetilde A}
\newcommand{\tB}{\widetilde B}
\newcommand{\tC}{\widetilde C}
\newcommand{\tD}{\widetilde D}

\newcommand{\ttpi}{\widetilde \pi}
\newcommand{\ttrho}{\widetilde \rho}
\newcommand{\tkappa}{\widetilde \kappa}
\newcommand{\tS}{\widetilde S}
\newcommand{\tDelta}{\widetilde \Delta}
\newcommand{\tbeta}{\widetilde \beta }
\newcommand{\tgamma}{\widetilde \gamma}

\begin{document}

\title{Essential spectrum of systems of singular differential equations}

\author{\hspace{-9mm}
Orif O.\ Ibrogimov$^{\circ}$, Heinz Langer$^{\sharp}$, Matthias Langer$^{\dagger}$, and Christiane Tretter$^{\circ}$\footnote{
The first, second and last author gratefully acknowledge the support of the German Research Foundation (DFG), grant no.\ TR368/6-2, and of the Swiss National Science Foundation (SNF), grant no.\ 200021-119826/1.
The third author acknowledges the support of the Engineering and Physical Sciences Research Council (EPSRC), grant no.\ EP/E037844/1, and
of the London Mathematical Society under Scheme 4, ref.~no.~4518.
The last three authors also thank the Nuffield Foundation, grant no.\ NAL/01159/G.
}}


\date{\emph{Dedicated to the memory of Professor B\'ela Sz\H{o}kefalvi-Nagy} \\[2mm] {\small \today} \vspace{-5mm}}

\maketitle

\begin{center}
\address{\phantom{I}$^{\circ}$Mathematisches Institut, Universit\"at Bern \\
Sidlerstr.~5, 3012 Bern, Switzerland}
\email{orif.ibrogimov@math.unibe.ch, tretter@math.unibe.ch}
\address{\phantom{I}$^{\sharp}$Institut f\"ur Analysis und Scientific Computing,
Vienna University of Technology \\
Wiedner Hauptstr.~8--10, 1040 Wien, Austria}
\email{hlanger@email.tuwien.ac.at}
\address{\phantom{I}$^{\dagger}$Department of Mathematics and Statistics, University of Strathclyde \\
26 Richmond Street, Glasgow G1 1XH, United Kingdom}
\email{m.langer@strath.ac.uk}
\end{center}

\begin{abstract}
In this paper we develop a new method to determine the essential spectrum of
coupled systems of singular differential equations.
Applications to problems from magnetohydrodynamics and astrophysics are given.
\\[1ex]
\end{abstract}

\section{Introduction}
\label{sec1}

Coupled systems of differential equations and related spectral problems are ubi\-quitous in stability problems in physics and engineering. Unlike scalar differential equations, their spectrum need not be discrete, even if the underlying domain is compact and all coefficients are smooth and regular. If the domain is compact, then the essential spectrum consists of one part, called regular part, which may be caused by cancellations in leading order coefficients of the formal determinant;
an additional part of essential spectrum, called singular part, may occur if the domain is no longer compact or the coefficients have singularities at the boundary. In view of stability or numerical approximations, knowledge of the location of the entire essential spectrum is of crucial importance.

This is reflected by the vast literature on this topic which includes papers by Grubb and Geymonat \cite{MR0435621}, Descloux and Geymonat \cite{MR580568}, Kako \cite{MR894249}, Ra\u{\i}kov \cite{MR1143439}, Atkinson, H.\ Langer, Mennicken, and Shkalikov \cite{MR1285306}, Beyer \cite{MR1347113}, H.\ Langer and M\"oller  \cite{MR1380711}, Konstantinov \cite{MR1706515}, Faierman and M\"oller \cite{MR1742577}, Mennicken, Naboko, and Tretter \cite{MR1887017}, Kurasov and Nabo\-ko \cite{MR2063547}, Marletta and Tretter \cite{MR2363469}, Qi and Chen \cite{MR2793253}, and many more papers by these authors and by others.

The new result in the present paper is a formula for the essential spectrum of systems of singular ordinary differential equations of the form
\begin{equation}
\label{system}
\begin{aligned}
 -(py_1')'+ qy_1  - (\ov{b} y_2)'+ \ov{c} y_2 - \la y_1 &= f_1, \\
 b y_1'+ c y_1 + d y_2 - \la y_2 &= f_2,
\end{aligned}
\end{equation}
either on the unbounded interval $[0,\infty)$, or on the finite interval~$(0,1]$ with coefficients having singularities at $0$. 
This formula (Theorem~\ref{ess_spec1} and Theorem~\ref{essspec01}, respectively) reveals the different nature of the two parts of the essential spectrum. The regular part depends only on the values of the coefficients $p$, $d$, and $b$ within the interval $[0,\infty)$. The singular part is determined by the 
limiting behaviour at $\infty$ of more involved combinations of all coefficients of the system and some of their derivatives.
Our assumptions allow for cases where the essential spectrum is unbounded from both sides.

Our main tool is a theorem on the essential spectrum of differential operators on $[0,\infty)$ with asymptotically constant coefficients due to Edmunds and Evans (see \cite{MR89b:47001}). Applied to the first Schur complements associated with the system \eqref{system}, it allows us to characterize the singular part of the essential spectrum.  The regular part is captured by means of Glazman's decomposition principle and a result of \cite{MR1285306}
derived by means of the second Schur complement.
Earlier approaches used the asymptotic Hain--L\"ust operator (which is the first Schur complement) (see e.g.\ \cite{MR2388939} and \cite{MR2169702}) or  
a transformation to canonical systems (see e.g.\ \cite{MR2793253}). The results therein (sometimes under slightly different assumptions) are all covered by our general
formula for the essential spectrum.

The paper is organized as follows. In Section \ref{sec2} we provide the necessary operator-theoretic framework
for systems \eqref{system} on $[0,\infty)$. In Section \ref{sec2a}, under some boundedness assumptions on the coefficients, we establish the relation between the essential spectra of the matrix differential operator given by \eqref{system} and its Schur complement. In Section \ref{sec2b}, assuming a certain limit behaviour at~$\infty$ of the coefficients of the Schur complement, we state and prove our main result (Theorem~\ref{ess_spec1}). In Section \ref{sec3}, under slightly stronger assumptions, we give 
an overview of the form of the essential spectrum which may consist of at most two possibly unbounded intervals.
In Section \ref{sec4} we establish the analogue of Theorem~\ref{ess_spec1} for systems of differential equations on $(0,1]$ with coefficients singular at $0$ by a transformation to systems on $[0,\infty)$ (Theorem~\ref{essspec01}). In Section~\ref{sec5} we show that our results 
unify, and shorten, the computation of the essential spectrum in the three papers mentioned above and in \cite{MR2074778}. 
In Section~\ref{sec6} we study the essential spectrum of a problem from \cite{MR1347113} arising in the stability analysis of spherically symmetric stellar equilibrium models.

We dedicate this paper to the memory of Prof.\ B\'ela Sz\H{o}kefalvi-Nagy, one of the pioneers of operator theory in the 20th century. H.L.\ had the great chance to meet him in 1959 and to enjoy his support and friendship for almost 40 years.



\section{Matrix differential operators on $[0,\infty)$}
\label{sec2}

In this section we consider $2\times 2$ matrix differential operators on $[0,\infty)$ of mixed order at most $2$ associated with a system of differential equations~\eqref{system} singular only at $\infty$.

We introduce the differential \vspace{-1mm} expressions
\begin{alignat*}{2}
  \tau_A \defequ -\dt p \dt + q, & \qquad
  \tau_B \defequ - \dt \ov{b}  + \ov{c}, \\
  \tau_C \defequ b \dt + c, & \qquad
  \tau_D \defequ d,
\end{alignat*}
with coefficient functions $p$, $q$, $b$, $c$, $d$ satisfying certain assumptions specified below.


Here, and in the sequel, for a subinterval $J\subset \RR$, $k\in \N_0$ and $K=\RR$ or $K={\mathbb C}$, we denote by  $C^k(J,K)$ the space of $k$-times continuously differentiable functions on $J$ with values in $K$; if $J$ is open, $C_0^k(J,K)$ is the space of all functions $f\in C^k(J,K)$ with compact support contained in $J$. If $k=0$, we write $C(J,K)$ and $C_0(J,K)$; if $K={\mathbb C}$, we also write $C^k(J)$ and $C^k_0(J)$.
Finally, ${\rm AC_{loc}}(J)$ denotes the space of locally absolutely continuous functions on $J$.  

\vspace{3mm}

\noindent
\textbf{Assumption (A).} \
$p,\,d\in C^{2}([0,\infty),\RR)$, $b \in C^2([0,\infty),\mathbb C)$, $c\in C^1([0,\infty),\mathbb C)$, $q\in C([0,\infty),\RR)$ 
with $p(x)>0$, $x\in[0,\infty)$.


\begin{remark}
For the results in Sections \ref{sec2} and \ref{sec2a}, the assumptions on $p$, $d$, and~$b$ can be weakened to $p,\,d\in C^1([0,\infty),\RR)$, $b \in C^1([0,\infty),\mathbb C)$; the stronger assumptions in (A) are needed for the main Theorem \ref{ess_spec1}.
\end{remark}

\medskip

Let $A_0$, $B_0$, $C_0$, $D_0$ be the operators in the Hilbert space $ L^2(0,\infty)$ induced by the differential expressions $\tau_A$, $\tau_B$, $\tau_C$, $\tau_D$ with domains
\[
  \cD(A_0):=C^2_0((0,\infty)), \quad \cD(B_0)=\cD(C_0):=C^1_0((0,\infty)), \quad \cD(D_0):=C_0((0,\infty)).
\]
In the Hilbert space $L^2(0,\infty)^2=L^2(0,\infty)\oplus L^2(0,\infty)$, we introduce the matrix differential operator
\begin{equation} 
\label{A0}
\begin{aligned}
  \cA_0 &\defeq \begin{pmatrix} A_0 & B_0 \\ C_0 & D_0 \end{pmatrix} \!=\!
  \begin{pmatrix} -\dt p \dt + q & - \dt \ov{b} + \ov{c} \\[2ex]
  b \dt + c & d \end{pmatrix}\!, \\ 
  \cD(\cA_0) &:= C_0^2((0,\infty))\oplus C_0^1((0,\infty)).
\end{aligned}  
\end{equation}

\begin{remark}
\label{Cinfcore}
It is not difficult to see that $C_0^\infty((0,\infty))\oplus C_0^\infty((0,\infty))$ is a core for $\cA_0$ since $p$, $b \in C^1([0,\infty))$ and $q$, $c$, $d\in C([0,\infty))$ by Assumption (A).
\end{remark}

It is well-known that closures and adjoints of block operator matrices in a Hilbert space product $\cH_1 \oplus \cH_2$ need not have domains decomposing into a corresponding direct sum; this may happen even for matrix differential operators on compact intervals with smooth coefficients. In general, their domains may contain conditions coupling the two components (see e.g.\ \cite[Theorems~2.2.14, 2.2.18]{CT}).

\begin{proposition}
\label{ha}
The operator $\cA_0$ in \eqref{A0} is symmetric in $L_2(0,\infty)^2$
with
\begin{align}
\label{domA0adj}
  &\cD(\cA_0^*) \!=\! \biggl\{\binom{y_1}{y_2}\in L_2(0,\infty)^2\colon
  y_1,\,py_1'+\ov{b}y_2 \in {\rm AC_{loc}}([0,\infty)), \\
  &\hspace{2.5cm} -\bigl(py_1'+\ov{b}y_2\bigr)'\!+qy_1+\ov{c}y_2,\,by_1'+cy_1+dy_2 \in L^2(0,\infty)\biggr\}, \hspace*{-5mm}\nonumber \\
\label{A0adj}
  &\cA_0^*\binom{y_1}{y_2}
  \!=\!\! \begin{pmatrix} -\bigl(py_1'+\ov{b}y_2\bigr)'+qy_1+\ov{c}y_2 \\[2ex]
  by_1'+cy_1+dy_2 \end{pmatrix},
\end{align}
and ${\rm dim\, ker}(\cA_0^*-\lambda)\le 2$ for $\la\in\mathbb C$.
\end{proposition}


The proof of this proposition is given in the Appendix.

\medskip

In this paper we use the following definition of the {\it essential spectrum} $\sess(T)$  of  
a densely defined closed linear operator $T$ in a Hilbert space $\cH$, 
which coincides with the definition of the essential spectrum $\sigma_{e3}(T)$ in \cite[Sections~I.3.4 and IX.1]{MR89b:47001}:
\[
  \sess (T) := \{ \la \in \CC : T-\la \text{ is not Fredholm}\};
\]
here $T$ is called  \emph{Fredholm} if $\ran(T)$ is closed and both  $\ker(T)$ and $(\ran(T))^\perp$ are finite-dimensional.

\smallskip

Since the deficiency numbers of $\cA_0$ are finite by Proposition \ref{ha}, we can use the following characterization of the essential spectrum. 

\begin{lemma}
\label{le.essunb}
\label{propA}
Let $T_{0}$ be a densely defined symmetric operator in a Hilbert space $\cH$ with finite deficiency numbers.
Then for every closed symmetric extension $T$ of $\,T_{0}$ in $\cH$ we have 
$\sess (T)=\sess (\ov{T_0}) \subset \RR$ 
and, for $\la\in\RR\setminus \sigma_{\rm p}(\ov{T_0})$,
\begin{align}
\label{sept13}
   \la \notin \sess(T)=\sess (\ov{T_0}) \iff  
   (T_0-\la)^{-1} \mbox{ is bounded on } \ran(T_0-\la).
\end{align}
If $\la\in\RR\setminus \big( \sigma_{\rm p}(\ov{T_0}) \cup  \sess(\ov{T_0})\big)$, then 
there exists a self-adjoint extension $T$ of $\,T_{0}$ such that $\la\in\rho(T)$.
\end{lemma}

\begin{proof}
Since the deficiency numbers of $T_{0}$ are finite, every closed symmetric extension $T$ of $\ov{T_0}$ is a finite-dimensional extension and hence $\sess (T)=\sess (\ov{T_0})$ by \cite[Corollary IX.4.2]{MR89b:47001}. Moreover, if $\la\in \CC\setminus\RR$, then $\ker(\ov{T_0}-\la)=\{0\}$,
$\ran(\ov{T_0}\!-\!\la)$ is closed, and $(\ran(\ov{T_0}\!-\!\la))^\perp$ is finite-dimensional, hence $\la\!\notin\!\sess(\ov{T_0})$.

If $\la\in\RR\setminus \big( \sigma_{\rm p}(\ov{T_0}) \cup \sess (\ov{T_0}) \big)$, then 
$\ran(\ov{T_0}-\la)$ is closed and
$(\ran(\ov{T_0}-\la))^\perp$ is finite-dimensional. Thus, by the closed graph theorem, $(\ov{T_0}-\la)^{-1}\supset (T_0-\la)^{-1}$ is bounded on $\ran(\ov{T_0}-\la) \supset \ran(T_0-\la)$.
This proves ``$\Longrightarrow$'' in \eqref{sept13} and implies, by \cite{MR0003466}\footnote{We discovered this 
reference by a remark in Riesz/Sz.-Nagy's monograph \cite[p.\,340]{MR1068530}.} applied to $\ov{T_0}$, the last claim. 

Vice versa, suppose that 
$(T_0-\la)^{-1}$ is bounded on $\ran(T_0-\la)$. Since $\ov{T_0}-\la$ is injective, it follows that  $\ran(\ov{T_0}-\la) = \ov{\ran(T_0-\la)}$. Hence $(\ov{T_0}-\la)^{-1}$ is bounded on $\ran(\ov{T_0}-\la)$. By \cite{MR0003466} this implies that the deficiency numbers of $\ov{T_0}$  are equal. Since they are finite by assumption and the mapping $z\mapsto \dim (\ran(\ov{T_0}-z))^\perp$ is locally constant on the set of points of regular type of $\ov{T_0}$, it follows that also $\dim (\ran(\ov{T_0} -\la))^\perp$ is finite and hence $\la\notin\sess(\ov{T_0})$.
\end{proof}


\medskip

By Proposition \ref{ha} and Lemma \ref{propA}, the essential spectrum of every closed symmetric extension $\cA$ of $\cA_0$ in \eqref{A0} coincides with $\sess(\ov{\cA_0})\subset\RR$.

In order to determine $\sess(\cA)$, we employ I.M.\,Glazman's decomposition principle (see \cite{MR0190800}).
To this end, for an arbitrary subinterval $J\subset [0,\infty)$ we denote by $\cA_J$ the closure of the symmetric operator $\cA_{0,J}$ in $L^2(J)^2$ generated by the restriction of $\cA_0$ to $C^2_0(J)\oplus C^1_0(J)$, i.e.
\begin{equation}
\label{clos1}
\cA_J=\ov{\cA_{0,J}}=\ov{\cA_0|_{C^2_0(J)\oplus C^1_0(J)}}; 
\end{equation}
in particular, $\cA_{(0,\infty)}\!=\!\ov{\cA_0}$.  
Then, for arbitrary  $t_0\in (0,\infty)$, the operator~$\cA$~in $L^2(0,\!\infty)^2$ is a 
finite-dimensional extension of the orthogonal sum $\cA_{(0,t_0)}\oplus\cA_{(t_0,\infty)}$  and hence
\begin{equation}
\label{glaz}
\sess(\cA)=\sess(\cA_{(0,t_0)})\,\cup\,\sess(\cA_{(t_0,\infty)}),
\end{equation}
see e.g.\ \cite[Corollary~IX.4.2 and Theorem~IX.9.3]{MR89b:47001}.

\section{Schur complement and essential spectrum }
\label{sec2a}


Schur complements are important tools to describe the spectrum, and especially the essential spectrum, of an operator matrix (see e.g.\ \cite[Sections~2.3, 2.4]{CT}). 
Here we use the first Schur complement $S(\la)$ of the operator matrix $\cA_0$ in~\eqref{A0}. 
Formally, $S(\la)$ is given by the scalar second order differential expression
\begin{align*}
  S(\la)\!&\defeq \tau_A-\la-\tau_B(\tau_D-\la)^{-1}\tau_C \\
  &=   -\dt p \dt + q - \la
  - \left( - \dt \ov{b}  + \ov{c} \right) \dfrac{1}{d\! -\! \la} \left( b \dt + c \right)
\end{align*} 
for $\la\in\CC\setminus d([0,\infty))$, where $d([0,\infty)):=\{d(t):t\in[0,\infty)\}$ is the range of the function $d$.
If we define, for $\la\in\CC\setminus d([0,\infty))$,
\begin{align} 
\label{pirho}
  \tpi(\cdot,\la)& \defeq p-\frac{|b|^2}{d-\la}, \qquad
  \trho(\cdot,\la) \defeq -\frac{2\Im (b\,\ov{c})}{d-\la}+\ii \frac{\partial}{\partial t}\tpi(\cdot,\la),\\[1mm]
\label{kappa}  
  \tka(\cdot,\la)&\defeq q - \la - \frac{|c|^2}{d-\la}+\dfrac{\partial}{\partial t}\left(\frac{\ov{b}\,c}{d-\la}\right),
\end{align}
it is not difficult to see that $S(\la)$ can be written in the following two ways:
\begin{align}
\label{Ssymm}
   S(\la)\!&=
             - \dfrac{\d}{\d t} \tpi(\cdot,\la)\dfrac{\d}{\d t}  -\frac{2\Im (b\,\ov{c})}{d-\la} \ii\dt +   \tka(\cdot,\la)\\
\label{S}
  \!&= -\tpi(\cdot,\la)\dfrac{\d^2}{\d t^2} +\trho(\cdot,\la)\, \ii\dt + \tka(\cdot,\la).
\end{align}

It is clear that, for $\la\in\RR\setminus d([0,\infty))$, the differential
expression $S(\la)$ is formally symmetric. For any subinterval $J\subset[0,\infty)$ and $\la\in \RR\setminus d(J)$, it induces a symmetric operator $S_{0,J}(\la)$ in $L^2(J)$ with domain $C^{2}_0(J)$; we denote the closure of $S_{0,J}(\la)$ by $S_J(\la)$:
\begin{equation}
\label{clos2}
  \cD(S_{0,J}(\la)):=C^2_0(J), \quad S_{0,J}(\la)u:=S(\la)u, \qquad
  S_J(\la) := \ov{S_{0,J}(\la)}.
\end{equation}


\medskip

The set of points $\la$ where the leading coefficient $\tpi(\cdot,\la)$ of $S(\la)$ vanishes, 
and hence classical Sturm--Liouville theory fails, plays an important role. The following lemma 
implies that this set is contained in the range of the function
\begin{equation}
\label{Delta}
 \Delta(t) \defeq d(t)-\frac{|b(t)|^2}{p(t)}, \quad t\in[0,\infty).
\end{equation}


\begin{lemma} 
\label{re.Pdelta}
\begin{enumerate}
\item[{\rm i)}] 
If $\la \notin d([0,\infty))$, then 
\begin{equation}\label{pidelta}
  \tpi(t,\lambda)
  =p(t)\frac{\Delta(t)-\lambda}{d(t)-\lambda}, \quad  t\in [0,\infty).
\end{equation}
\item[{\rm ii)}] For a subinterval $J\subset[0,\infty)$, the range $\Delta(J)$ is given by 
\begin{align*}
  \Delta(J)=&\bigl\{d(t)\!:t\in\!J \text{ with } b(t)=0\bigr\} 
  \cup\bigl\{\lambda\in\RR\!:\exists\,t\in\!J\text{ with } \tpi(t,\lambda)=0\bigr\}.
\end{align*}
\end{enumerate}
\end{lemma}

\begin{proof}
i) Formula \eqref{pidelta} follows if we note that, by the definitions of $\pi(\cdot,\la)$ and $\Delta$ in \eqref{pirho} and \eqref{Delta},
\[
 \tpi(t,\lambda)=p(t)-\frac{p(t)\bigl(d(t)-\Delta(t)\bigr)}{d(t)-\lambda},  \quad  t\in [0,\infty).
\]
ii) We have $\lambda\in\Delta(J)$ if and only if there is a $t\in J$ with 
  $p(t)\bigl(d(t)-\lambda\bigr)=|b(t)|^2$.
Since $p(t)\ne0$, this holds if and only if $b(t)=0$ and $\lambda=d(t)$, or if $b(t)\ne0$ and
$p(t)=|b(t)|^2/(d(t)-\lambda)$, i.e.\ $\tpi(t,\lambda)=0$. Note that the latter automatically implies that $b(t)\ne0$ because $p(t)\ne0$.
\end{proof}


\begin{remark}
\label{schurdef}
From Assumption (A) and \eqref{Ssymm}, \eqref{pidelta} it follows that the differential expression $S(\la)$  satisfies the conditions \cite[(10.3)]{MR89b:47001}  of the existence and uniqueness theorem  \cite[Theorem~III.10.1]{MR89b:47001} on a subinterval $J\!\subset\! [0,\infty)$ for $\la\!\in\!\CC\setminus (d(J) \cup \Delta(J))$; in \vspace{-1mm} fact,
\begin{align*}
 {\rm (i)}  \ \ &\,\pi(\cdot,\la) \ne 0, \ \frac 1{\pi(\cdot,\la)} \in L^1_{\rm loc}(J) ,\\
 {\rm (ii)} \ \ &-\dfrac{2\Im (b\,\ov{c})}{d-\la} \dfrac 1{\pi(\cdot,\la)} = -\dfrac{2\Im (b\,\ov{c})}{p(\Delta-\la)} \in {\rm AC_{loc}}(J), \quad \kappa(\cdot,\la) \in L^1_{\rm loc}(J). 
\end{align*}
Since, in addition, $-\dfrac{2\Im (b\,\ov{c})}{d-\la} \in {\rm AC_{loc}}(J)$, the symmetric operator $S_{0,J}(\la)$ for $\la\in\CC\setminus (d(J) \cup \Delta(J))$ also satisfies the conditions of \cite[Theorem III.10.7]{MR89b:47001} and hence the deficiency numbers of $S_{0,J}(\la)$ are~$\le 2$.
\end{remark}

\smallskip

It is an immediate consequence of \cite{MR1285306} that 
the range of the function $\Delta$ in \eqref{Delta} belongs to the essential spectrum of $\cA$; 
later it will be called the regular part of the essential spectrum.

\begin{proposition}
\label{regpart}
For every closed symmetric extension $\cA$ of the operator~$\cA_0$ in \eqref{A0}, we have
\[
 \ov{\Delta([0,\infty))}\subset\sess\bigl(\cA);
\]
in particular, $\,\ov{\Delta([0,\infty))}=\RR$ implies that $\sess\bigl(\cA)=\RR$.
\end{proposition}

\begin{proof}
If $\la\in\Delta([0,\infty))$, there exists a $t_0\in(0,\infty)$ so that
$\la\in\Delta([0,t_0])$. By \cite[Theorem~4.5]{MR1285306} and Glazman's decomposition principle \eqref{glaz}, we have
\begin{align}
\label{ALMS}
  \la \in \Delta([0,t_0]) = \sess(\cA_{(0,t_0)}) \subset \sess(\cA).
\end{align}
Since the essential spectrum is closed, the claimed inclusion follows. 
The last claim is immediate because $\sess(\cA)=\sess(\ov{\cA_0})\subset \RR$.
\end{proof}


In order to characterize the essential spectrum of the restrictions $\cA_{(t_0,\infty)}$ in terms of the Schur complement, we need further assumptions on the coefficient functions in \eqref{A0} and~\eqref{S}. To formulate them, the following notation will be convenient.
%

\begin{notation}
\label{not-inf} 
For functions $f:[0,\infty) \to \CC$ and $g:[0,\infty)\to \RR$ having possibly improper limits at $\infty$, we write
\[
  f\linf:= \lim_{t\to\infty} f(t) \in \CC \cup \{\infty\}, \quad g\linf:= \lim_{t\to\infty} g(t) \in \RR \cup \{\pm\infty\};
\]
if $f\linf$ and $g\linf$ are finite, they will also be used in arithmetic calculations.
\end{notation}

%

\medskip

\noindent
\textbf{Assumption (B).} \
(B1) The possibly improper limit $\dinf$ exists.
\begin{enumerate}
\item[(B2)]
There exist constants $\beta$, $\gamma>0$ such that 
\[
  |b(t)|\le \beta(|d(t)|+1), \ \ |c(t)|\le \gamma(|d(t)|+1) \ \text{  for all } \ t\in [0,\infty).
\]   
\item[(B3a)\!]
For some (and hence all, see Remark \ref{B3a} below) $\la\in \RR \setminus \{\dinf\}$ there exists $t_{\la} \in [0,\infty)$ 
such that $\la\notin d([t_{\la},\infty))$ and
\[
 \tpi(\cdot, \la) \ \text{ is a bounded function on } \ [t_{\la},\infty).
\]
\vspace{-5mm}
\item[(B3b)\!]
For every $\la\in \RR \setminus \big( \overline{\Delta([0,\infty))} \cup \{\dinf\} \big) $ there exists a $t_\la \in [0,\infty)$ 
such that $\la\notin d([t_\la,\infty))$ and
\[
   \trho(\cdot,\la), \ \tka(\cdot,\la), \ \dfrac 1{\tpi(\cdot, \la)} \ \text{ are bounded functions on } \ [t_\la,\infty).
\]
\end{enumerate}


\begin{remark}
If  Assumption (B1) holds, then $\ov{\Delta([0,\infty))}=\RR$ is only possible if $\dinf = +\infty$; 
otherwise, if $\dinf < +\infty$, it is immediate that $\Delta([0,\infty))$ is bounded from above since $p>0$ by Assumption (A).
\end{remark}
 
\begin{remark}
\label{B3a}
If (B3a) holds for some $\la_0 \in \RR \setminus\{ \dinf\}$, 
then the~identity 
\begin{equation}
\label{pila0}
 \tpi(\cdot,\la) -\tpi(\cdot,\la_0) = (\la_0-\la) \frac{|b|^2}{(d-\la)(d-\la_0)} 
\end{equation}
and Assumption (B2) on $b$ ensure that (B3a) holds for all $\la \in \RR \setminus\{ \dinf\}$ if we choose $t_\la \ge t_{\la_0}$.
\end{remark}

\begin{lemma} 
\label{le.AS1}
Suppose that $\Delta([0,\infty))\ne \RR$, and
let $\la\in \RR \setminus \big( \overline{\Delta([0,\infty))} \cup \{\dinf\} \big) $. Then there exists a $t_\la \in (0,\infty)$ such that 
$\la\notin  \ov{d([t_\la,\infty))}$
and, with this $t_\la$,
\begin{equation}
\label{Aequivschur}
  \lambda\in\sess(\cA_{(t_\la,\infty)}) \;\Longleftrightarrow\;
  0\in\sess(S_{(t_\la,\infty)}(\la)).
\end{equation}
\end{lemma}

\begin{proof}
As $\Delta$ is continuous due to Assumption (A), the condition $\Delta([0,\infty))\ne \RR$ implies that 
$\Delta([0,\infty))$ is a proper subinterval of $\RR$. 
The existence of $t_\la$ with $\la\notin  \ov{d([t_\la,\infty))}$ is immediate from Assumption (B1) on the existence of the 
(possibly improper) limit $\dinf$.

In order to prove the equivalence \eqref{Aequivschur},
we fix $\la \in \RR \setminus \big(\overline{\Delta([0,\infty))} \cup  \ov{d([t_\la,\infty))}\big)$ and
we abbreviate $J \defeq (t_\la,\infty)$.  

By the definitions \eqref{clos1}, \eqref{clos2}, we have $\cA_{J}=\ov{\cA_{0,J}}$, $S_J(\la) = \ov{S_{0,J}(\lambda)}$ and hence
we have to prove that 
\begin{equation} \label{hot}
  \lambda\in\sess(\ov{\cA_{0,J}}) \quad\Longleftrightarrow\quad
  0\in\sess\bigl(\ov{S_{0,J}(\lambda)}\bigr).
\end{equation}
First we show that Lemma~\ref{le.essunb} 
can be applied to the operators $\cA_{0,J}$ and $S_{0,J}(\lambda)$.

By Remark \ref{schurdef}, the operator $S_{0,J}(\la)$ is symmetric with deficiency numbers~$\le 2$; moreover, 
$0\notin\spt\bigl(\ov{S_{0,J}(\lambda)}\bigr)$ since every
eigenfunction $y$ of $\ov{S_{0,J}(\lambda)}$ must satisfy the second order differential
equation $S(\lambda)y=0$ and the boundary conditions $y(t_\la)=y'(t_\la)=0$
and hence would have to vanish identically.
By Proposition~\ref{ha}, the operator $\cA_{0,J}$ is symmetric with deficiency numbers $\le 2$.
Further, it is not difficult to check that $\la\in\spt\bigl(\ov{\cA_{0,J}}\bigr)$ 
implies that $0\in\spt\bigl(\ov{S_{0,J}(\lambda)}\bigr)$, which was excluded above. Hence $\la\notin\spt\bigl(\ov{\cA_{0,J}}\bigr)$.

Now Lemma~\ref{le.essunb} implies that $\sess\bigl(\ov{\cA_{0,J}}\bigr) \subset \RR$, $\sess\bigl(\ov{S_{0,J}(\lambda)}\bigr)\subset \RR$ and that we can use the characterization \eqref{sept13} for both sets.

``$\Longleftarrow$" in \eqref{hot}:
Assume that $\lambda\notin\sess\bigl(\ov{\cA_{0,J}}\bigr) $. 
By \eqref{sept13}, the claim is proved if we show that $S_{0,J}(\la)^{-1}$ is bounded on 
$\ran S_{0,J}(\la)$. Let $f\in\ran S_{0,J}(\la)$, $f= S_{0,J}(\lambda)u$ with $u\in  C_0^2(J)$.
By Assumption (A), we have
\[
  v \defeq -\frac1{d-\lambda}\Bigl(b\frac{\d}{\d t}+c\Bigr)u \ \in C_0^1(J).
\]
Hence $(u,v)^{\rm t} \in C^2_0(J) \oplus C^1_0(J) = \cD(\cA_{0,J})$ and, by the definition of $v$,
\[
  \bigl(\cA_{0,J}-\lambda\bigr)\vect{u}{v}=\vect{S_{0,J}(\lambda)u}{0}=\vect{f}{0}.
\]
Since $\lambda\!\notin\!\sess\bigl(\ov{\cA_{0,J}}\bigr)$, we know that $(\cA_{0,J}\!-\!\la)^{-1}\!$ is bounded on $\ran(\cA_{0,J}\!-\!\la)$ by \eqref{sept13}.
If $P_1$ denotes the projection onto the first component in $L^2(J)^2$, then 
\[ 
  S_{0,J}(\lambda)^{-1}f = u = P_1 (\cA_{0,J}-\la)^{-1} \begin{pmatrix}{f}\\{0}\end{pmatrix}. 
\]
Hence $S_{0,J}(\la)^{-1}$ is bounded on $\ran S_{0,J}(\la)$ and so $0\notin\sess\bigl(\ov{S_{0,J}(\lambda)}\bigr)$ by \eqref{sept13}.

``$\Longrightarrow$" in \eqref{hot}:
Suppose that $0\notin\sess\bigl(\ov{S_{0,J}(\lambda)}\bigr)$. Then, by Lemma~\ref{le.essunb}, there exists a
self-adjoint extension $\wt S_{J}(\lambda)$ of $S_{0,J}(\lambda)$ such that $0\in\rho(\wt S_{J}(\lambda))$.
By \eqref{sept13}, the claim is proved if we show that $(\cA_{0,J}-\la)^{-1}$ is bounded on $\ran(\cA_{0,J}-\la)$.
Let $(f,g)^{\rm t} \in \ran(\cA_{0,J}-\la)$, 
\begin{equation}
\label{nov18}
\begin{pmatrix}{f}\\{g}\end{pmatrix} =(\cA_{0,J}-\lambda)\begin{pmatrix}{u}\\{v}\end{pmatrix} 
\ \iff \
\begin{aligned}
  (A_0-\lambda)u+B_0 v &= f \\
  C_0 u+(D_0-\lambda)v &= g
\end{aligned}
\end{equation}
with $(u,v)^{\rm t} \in \cD(\cA_{0,J}) = C^2_0(J) \oplus C^1_0(J)$. The latter and Assumption (A) imply that 
$(D_0-\la)^{-1}C_0u \in C^1_0(J)$, and the second equation shows that $(D_0-\la)^{-1}g=(D_0-\la)^{-1}C_0u+v\in C_0^1(J)$.
Solving the second equation for $v$, we can thus substitute $v=-(D_0-\lambda)^{-1}C_0u+(D_0-\lambda)^{-1}g$ into the first equation to obtain
\[
   (A_0-\lambda)u-B_0(D_0-\lambda)^{-1}C_0u=f-B_0(D_0-\lambda)^{-1}g.
\]
Since the left hand side equals $S_{0,J}(\la)u = \wt S_J(\la) u$ and $\wt S_J(\la)$ is boundedly invertible, it follows that
\[
  u=\wt S_J(\lambda)^{-1}f-\wt S_J(\lambda)^{-1}B_0(D_0-\lambda)^{-1}g.
\]
Inserting this back into the above formula for $v$, we find that
\[
  v=-(D_0-\lambda)^{-1}C_0\wt S_J(\lambda)^{-1}\!f+(D_0-\lambda)^{-1}g
  +(D_0-\lambda)^{-1}C_0\wt S_J(\lambda)^{-1}\!B_0(D_0-\lambda)^{-1}g.
\]
If we use the decomposition $\wt S_J(\lambda)^{-1} \!\!=\! |\wt S_J(\lambda)|^{-1/2} \sign \wt S_J(\lambda) |\wt S_J(\lambda)|^{-1/2}$,
observe that $B_0 = C_0^*|_{C_0^1(J)}$ and introduce the operator
\begin{align*} 
  F(\lambda) \defeq (D_0-\lambda)^{-1}C_0|\wt S_J(\lambda)|^{-\frac 12},
\end{align*}
it follows that
\[
  F(\la)^* \supset |\wt S_J(\lambda)|^{-\frac 12} B_0 (D_0-\lambda)^{-1}.  
\] 
Hence the above relations for $u$ and $v$ can be written in matrix form as
\[
 \begin{pmatrix}u\\[4mm] v   \!\end{pmatrix}\!\!=\!\!
  \begin{pmatrix}\!
  \wt S_J(\lambda)^{-1}
  \!&\! -|\wt S_J(\lambda)|^{-\frac 12}\sign\!\wt S_J(\lambda)\,F(\lambda)^* \\[2ex]
  -F(\lambda)\sign\!\wt S_J(\lambda)\,|\wt S_J(\lambda)|^{-\frac 12}
  \!&\! (D_0\!-\!\lambda)^{-1}\!+\!F(\lambda)\sign\!\wt S_J(\lambda)\,F(\lambda)^*
  \!\end{pmatrix}
  \!\!\!\begin{pmatrix}\!f\\[4mm]  g \end{pmatrix}\!.
\]
By \eqref{nov18}, the above block operator matrix is equal to $(\cA_{0,J}-\la)^{-1}$.
If we show that $F(\la)$ is a bounded operator in $L^2(J)$, then all entries in this operator matrix are bounded in
$L^2(J)$ and hence $(\cA_{0,J}-\lambda)^{-1}$ is bounded on $\ran(\cA_{0,J}-\la)$, which implies that $\lambda\notin\sess(\ov{\cA_{0,J}})$ by \eqref{sept13}.

By Assumption (B2) and since $\la\notin\ov{d(J)}$, we have 
\[
  \left| \frac{b(t)}{d(t)-\la} \right| \le \beta \frac{|d(t)|+1}{|d(t) - \la|},  \quad 
  \left| \frac{c(t)}{d(t)-\la} \right| \le \gamma \frac{|d(t)|+1}{|d(t) - \la|},
  \quad t\in \ov{J}.
\]
Therefore the functions on the left hand sides are bounded on $\ov{J}$ and hence
the closure of the first order differential operator
\[
  (D_0-\lambda)^{-1}C_0=\frac{b}{d-\lambda}\frac{\d}{\d t}+\frac{c}{d-\lambda}
\]
is bounded from the first order Sobolev space $W^{1,2}(J)$ to $L^2(J)$. 

By Assumptions (B3a), (B3b), the functions $\tpi(\cdot,\lambda)$, $\dfrac 1{\tpi(\cdot,\la)}$, $\trho(\cdot,\lambda)$, and $\tka(\cdot,\lambda)$,
are bounded on $J$ and hence the self-adjoint realization $\wt S_J(\la)$ of the differential expression
\[
  S(\la) = \tpi(\cdot,\la)
  \Big( - \dfrac{\d^2}{\d t^2} +\dfrac{\trho(\cdot,\la)}{\tpi(\cdot,\la)}\, \ii\dfrac{\d}{\d t} + \dfrac{\tka(\cdot,\la)}{\tpi(\cdot,\la)} \Big) 
\]
is bounded from the second order Sobolev space $W^{2,2}(J)$ to $L^2(J)$.
Thus the domain $\cD(|\wt S_J(\lambda)|^{1/2})$ (which is equal to
the form domain of $\wt S_J(\lambda)$) is either
$W^{1,2}(J)$ or $W^{1,2}_0(J)$ (the closure of $C_0^\infty(J)$ in $W^{1,2}(J))$, and the operator
$|\wt S_J(\lambda)|^{-1/2}$ is bounded from $L^2(J)$ to
$W^{1,2}(J)$. 

Altogether this implies that  $F(\la) \subset \ov{(D_0-\lambda)^{-1}C_0} |\wt S_J(\lambda)|^{-1/2}$ is a bounded operator in $L^2(J)$. 
\end{proof}


\section{Main result}
\label{sec2b}

The description of the singular part of the essential spectrum in our main result depends only on the limits at $\infty$ of certain functions formed out of the coefficients of the original operator matrix $\cA_0$. 
Here the following assumption is crucial.

\medskip

\noindent
\textbf{Assumption (C).}
\
For $\la\in\RR\setminus \big( \ov{\Delta([0,\infty))} \cup \{ \dinf\} \big)$ the following limits exist and are finite:
\[
  \Big( \dfrac{\trho(\cdot,\la)}{\tpi(\cdot,\la)} \Big)\blinf =\lim_{t\to\infty}\dfrac{\trho(t,\la)}{\tpi(t,\la)},\quad 
  \Big( \dfrac{\tka(\cdot,\la)}{\tpi(\cdot,\la)} \Big)\blinf  =\lim_{t\to\infty}\dfrac{\tka(t,\la)}{\tpi(t,\la)}.
\]

\smallskip

Observe that,  if $\la\ne \dinf$, the functions $\tpi(t,\la)$, $\trho(t,\la)$, and $\tka(t,\la)$ given by \eqref{pirho}, \eqref{kappa}  
are defined for all sufficiently large $t\in [0,\infty)$. 

\smallskip

We emphasize that we do not require that the limits $\tpi(\cdot,\la)\linf$, $\trho(\cdot,\la)\linf$, $\tka(\cdot,\la)\linf$ of the denominator and the numerators in Assumption (C) exist separately; this particular case will be studied in the next~section.

\begin{remark}
\label{last!}
It is easy to see that Assumptions (B3a) and (C) imply (B3b).
\end{remark}

\smallskip

Although the coefficient functions $\trho(\cdot,\la)$ and $\tka(\cdot,\la)$ may be complex-valued for real $\la$, 
the fact that they originate from the symmetric differential expression $S(\la)$ (for real $\la$) and the existence 
of the limits in Assumption (C) above have the following implications.


\begin{lemma}
\label{rem-orif}
If Assumptions {\rm (A)}, {\rm (B)}, and {\rm (C)} hold, then
for $\la\in\RR\setminus \big( \ov{\Delta([0,\infty))} $ $\cup \{ \dinf\} \big)$ the following limits exist, are finite and satisfy
\begin{equation}
\label{llimits}
\Big( \dfrac{ \frac{\partial}{\partial t} \tpi(\cdot,\la)}{\tpi(\cdot,\la)} \Big)\blinf=0, \quad
\Big( \dfrac{\trho(\cdot,\la)}{\tpi(\cdot,\la)} \Big)\blinf \in \RR, \quad
\Big( \dfrac{\tka(\cdot,\la)}{\tpi(\cdot,\la)} \Big)\blinf \in \RR. 
\end{equation}
%
\end{lemma}

\vspace{-4mm}

\begin{proof}
Let $\la\in \RR \setminus \big( \overline{\Delta([0,\infty))} \cup \{\dinf\} \big)$ be arbitrary.

By Assumptions (B3a), (B3b), there is a $t_{\lambda}>0$ such that $\la\notin d([t_\la,\infty))$ and the functions 
$\pi(t,\lambda)$, $\frac{1}{\pi(t,\lambda)}$ are bounded on $[t_{\lambda}, \infty)$.
It follows from the particular form of the coefficient $\trho(\cdot,\la)$ in \eqref{pirho} and Assumption (C) that the~limits 
\begin{align}
\label{cla}
\gamma(\lambda):=\Big( \dfrac{ \frac{\partial}{\partial t} \tpi(\cdot,\la)}{\tpi(\cdot,\la)} \Big)\blinf
= \Imag \Big( \dfrac{\trho(\cdot,\la)}{\tpi(\cdot,\la)} \Big)\blinf
\end{align}
coincide, exist and are finite.
If $\gamma(\lambda)>0$, then there is a $t'_{\lambda} \in [t_{\lambda},\infty)$ with $\frac{\frac{\D}{\D t}\pi(t,\lambda)}{\pi(t,\lambda)}\geq \frac{1}{2}\gamma(\lambda)$ for $t\in [t'_{\lambda}, \infty)$. Hence Gronwall's Lemma implies that 
\begin{align*}
\pi(t,\lambda)\geq \pi(t'_{\lambda},\lambda)\cdot\exp\Bigl(\frac{1}{2}(t-t'_{\lambda})\gamma(\lambda)\Bigr), \quad t\in [t'_{\lambda}, \infty),
\end{align*}
contradicting the fact that $\pi(t,\lambda)$ is bounded on $[t_{\lambda},\infty)$. 
Similarly, if $\gamma(\lambda)<0$, we obtain a contradiction to the boundedness of $\frac{1}{\pi(t,\lambda)}$ on $[t_{\lambda}, \infty)$. 
Therefore, $\gamma(\lambda)=0$  and hence by \eqref{cla} the claims for the first two limits in \eqref{llimits} are~proved.

The particular form of the coefficients $\trho(\cdot,\la)$, $\tka(\cdot,\la)$ in \eqref{pirho}, \eqref{kappa} and the fact that $\pi(\cdot,\la)$ is real-valued imply that, for $t\in [t_\la,\infty)$,
\begin{align}
\label{imka}
  \hspace{-4mm}
  \Imag \frac{\tka(t,\lambda)}{\tpi(t,\la)} & = \dfrac 12 \frac{\frac{\D}{\D t} \Real \trho(t,\lambda)}{\tpi(\cdot,\la)} \\
\label{imka1}
  & = \frac 12 \bigg( \frac{\D}{\D t} \bigg(\!\Real\Bigl(\frac{\trho(t,\lambda)}{\tpi(t,\lambda)}\Bigr)\bigg) \!+\! 
  \Real \Bigl(\frac{\trho(t,\lambda)}{\tpi(t,\lambda)} \Bigr) \frac{\frac{\D}{\D t} \tpi(t,\lambda)}{\tpi(t,\la)}
  \bigg). \hspace{-4mm}
\end{align}
According to Assumption (C) and \eqref{cla}, the limits at $\infty$ of the function on the left hand side of \eqref{imka} and of the two factors of the second term  in \eqref{imka1} exist. Hence also the limit of the first term in \eqref{imka1} exists and, by the first equality in~\eqref{llimits},
\begin{align}
\label{imka2}
 \Imag \Big( \dfrac{\tka(\cdot,\la)}{\tpi(\cdot,\la)} \Big)\blinf 
 \!\!=  \dfrac 12  \bigg( \frac{\D}{\D t} \left(\Real\Bigl(\frac{\trho(\cdot,\lambda)}{\tpi(\cdot,\lambda)}\Bigr)\bigg) \right)\blinf\!\!.
\end{align}
Since by Assumption (C) the limit $\Big(\Real\Bigl(\frac{\trho(\cdot,\lambda)}{\tpi(\cdot,\lambda)}\Bigr)\Big)\blinf$ exists, the limit on the right hand side of \eqref{imka2} has to be equal to $0$, which proves the third claim. 
\end{proof}


The following theorem, which is the main result of this paper, contains an explicit
description of the essential spectrum of the (closure of the) matrix differential operator in \eqref{A0}. 
It consists of a regular part determined by the behaviour of the coefficients within the interval $[0,\infty)$,
and of a singular part determined by the behaviour  at $\infty$ of certain functions of the coefficients and some of their derivatives. 

\medskip


\begin{theorem} 
\label{ess_spec1}
Suppose that Assumptions {\rm (A)}, {\rm (B)}, and {\rm (C)} are satisfied. 
Then the essential spectrum of every closed symmetric extension $\cA$ of the operator~$\cA_0$ in \eqref{A0} is given~by
\[
  \sess\bigl(\cA\bigr) \setminus\{\dinf\}
  = \big( \sessreg\bigl(\cA\bigr)\cup \sesssing\bigl(\cA\bigr)   \big) \, \setminus \{\dinf\}
\]
where
\begin{align*}
  \sessreg\bigl(\cA\bigr)
  &\!:=\! \ov{\Delta\bigl([0,\infty)\bigr)}, \\ 
  \sesssing\bigl(\cA\bigr)
  &\!:=\!\! \biggl\{\!\la\!\in\!\RR\!\setminus\! \big( \ov{\Delta([0,\infty))} \!\cup\! \{\dinf\} \big)\!:\!
    \bigg( \dfrac{\trho(\cdot,\la)}{\tpi(\cdot,\la)} \bigg)\blinf^2\!\!-4 \bigg( \dfrac{\tka(\cdot,\la)}{\tpi(\cdot,\la)} \bigg)\blinf \!\!\!\ge 0
    \biggr\}
\vspace{-5mm}    
\end{align*}
with $\Delta$ given by \eqref{Delta} and $\tpi(\cdot,\la)$, $\trho(\cdot,\la)$, and $\tka(\cdot,\la)$ defined as in \eqref{pirho}, \eqref{kappa}.
\end{theorem}



\begin{proof}
If $\Delta([0,\infty))= \RR$, then $\sess(\cA) = \RR$ by Proposition \ref{regpart} and there is nothing left to prove.

Now suppose that $\Delta([0,\infty))\ne \RR$ and let $\la \notin \big( \ov{\Delta([0,\infty))} \cup \{ \dinf\} \big)$.
Then, by Lemma~\ref{le.AS1}, there exists $t_\la \in (0,\infty)$ such that 
$\la\notin \big( \ov{\Delta([0,\infty))} \cup \ov{d([t_\la,\infty))}\big)$.
In particular, by \eqref{ALMS}, $\la \notin \overline{\Delta([0,t_\la])} 
= \sess(\cA_{(0,t_\la)})$. Hence Glazman's decomposition principle and \eqref{Aequivschur} imply that
\begin{align}
\label{sess-equiv}
 \la \in \sess(\cA) \  \iff \ \la \in \sess(\cA_{(t_\la,\infty)}) 
                       \iff \ 0 \in \sess(S_{(t_\la,\infty)}(\la)).
\end{align}
%
By Assumptions (B3a), (B3b), the functions $\tpi(\cdot,\la)$ and $\!\frac 1{\tpi(\cdot,\la)}$ are bounded on $[t_\la,\infty)$
and hence
\[
  0\in \sess(S_{(t_\la,\infty)}(\la)) \ \iff \ 0\in  \sess\Big( \frac 1{\tpi(\cdot,\la)} S_{(t_\la,\infty)}(\la) \Big).
\]
By Assumptions (A) and (C), the differential operator
\begin{equation}
  \dfrac 1{\tpi(\cdot,\la)}S_{(t_\la,\infty)}(\la)
  = - \dfrac{\d^2}{\d t^2} +\dfrac{\trho(\cdot,\la)}{\tpi(\cdot,\la)}\, \ii\dfrac{\d}{\d t} + \dfrac{\tka(\cdot,\la)}{\tpi(\cdot,\la)}
 \end{equation}
satisfies the conditions of \cite[Corollary~IX.9.4]{MR89b:47001} with $m\!=\!2$, $a_2\!=\!-1$, $a_1\!=\!\frac{\trho(\cdot,\la)}{\tpi(\cdot,\la)}$, and 
$a_0=\frac{\tka(\cdot,\la)}{\tpi(\cdot,\la)}$, except for $a_1' \in L_{\infty}(I)$ in \cite[p.\ 445, (iii)]{MR89b:47001}. However, a closer look at the proofs in
\cite[Section IX.9]{MR89b:47001} shows that it is enough to assume that $a_1' \in L_{\infty,{\it loc}}(I)$ therein.\footnote{We thank W.D.\ Evans for this personal communication.} The latter is guaranteed,  in our case, by Assumption (A) which ensures that $a_1'$ is continuous. 
Therefore \cite[(9.19)]{MR89b:47001} for $k=3$ applies and yields that
\begin{equation}
\label{qe}
  \sess\left(\dfrac 1{\tpi(\cdot,\la)}S_{(t_\la,\infty)}(\la)\right)
  =\left\{\xi^2+\Big( \dfrac{\trho(\cdot,\la)}{\tpi(\cdot,\la)} \Big)\blinf \xi
   + \Big( \dfrac{\tka(\cdot,\la)}{\tpi(\cdot,\la)} \Big)\blinf : \xi\in\RR\right\}.
\end{equation}
Hence 
\begin{align}
  0\in\sess(S_{(t_\la,\infty)}(\la)) \iff \exists\,\xi \in \RR: \xi^2+\Big( \dfrac{\trho(\cdot,\la)}{\tpi(\cdot,\la)} \Big)\blinf \xi
   + \Big( \dfrac{\tka(\cdot,\la)}{\tpi(\cdot,\la)} \Big)\blinf =0.
\end{align}  
By Lemma \ref{rem-orif}, the coefficients of the above quadratic equation are both real and hence a real solution $\xi$ exists if and only if
the corresponding discriminant $\Dis(\la)$ is non-negative,~i.e.
\begin{equation}
\label{def-discr}
  \Dis(\la)
  := \Big( \dfrac{\trho(\cdot,\la)}{\tpi(\cdot,\la)} \Big)\blinf^2\!\!-4 \Big( \dfrac{\tka(\cdot,\la)}{\tpi(\cdot,\la)} \Big)\blinf  \ge 0. 
  \qedhere
\vspace{2mm}  
\end{equation}
\end{proof}

\section{The form of the essential spectrum}
\label{sec3}

In this section we describe the regular part $\sessreg\bigl(\cA\bigr)$  and the singular part $\sesssing\bigl(\cA\bigr)$
of the essential spectrum of $\cA$. 
For the singular part we consider the special case of Assumption (C) where the limits of the functions 
$\tpi(\cdot,\la)$, $\trho(\cdot,\la)$, and $\tka(\cdot,\la)$ exist separately.   



\smallskip

Throughout this section we assume that Assumptions (A), (B1), (B2), (B3a), and (B3b) on the functions $p$, $q$, $b$, $c$, and $d$ are satisfied;
in particular, the possibly improper limit $\dinf=\lim_{t\to\infty}\,d(t)$ exists by (B1). 

\medskip

\noindent
{\bf 5.1. The regular part.}
By Theorem \ref{ess_spec1} the regular part of the essential spectrum of $\cA$ is the 
closure of the range of the function $\Delta$ defined in \eqref{Delta},
\[
   \Delta(t) \defeq d(t)-\frac{|b(t)|^2}{p(t)}, \quad t\in[0,\infty).
\]
The range of $\Delta$ is an interval since $d$, $b$, $p$ are continuous and $p>0$ by (A).
Let
\begin{equation}
\label{depm}
  \de_-:=\inf_{t\in[0,\infty)} \!\Delta(t), \quad
  \de_+:=\sup_{t\in[0,\infty)} \!\Delta(t).
\end{equation}

\begin{proposition} 
\label{prop2:regpart}
Assume that Assumptions {\rm (A)}, {\rm (B1)}, {\rm (B2)}, {\rm (B3a)}, and {\rm (B3b)} are satisfied. 

\begin{enumerate}
\item[{\rm i)}]
If $\,\dinf \in \RR$, then $\sessreg(\cA)$ is bounded from above, i.e.\ $\delta_+\in\RR$, and 
\[
 \sessreg(\cA)=
 \begin{cases}
 \hspace{3mm} [\delta_-,\delta_+] \textnormal{ with } \delta_-\!\in\RR & \textnormal{if } \ \liminf\limits_{t\to\infty}\, p(t)>0,\\
 (-\infty,\delta_+] & \textnormal{if } \ \liminf\limits_{t\to\infty}\,  p(t)=0.
 \end{cases}
\]
\item[{\rm ii)}] If $\,\dinf=+\infty$, then 
\[
  \sessreg(\cA)  = 
  \
  [\de_-,\de_+] 
  \textnormal{ \ \ if } \,\liminf_{t\to\infty} \Big( \dfrac{p}{d}\Big)(t)>0 
\]  
and, in the case  $\Delta([0,\infty)) \ne \RR$,
\begin{align*}
  \sessreg(\cA) & = 
  \begin{cases} 
  [s,+\infty) 
  & \textnormal{if } \,\liminf\limits_{t\to\infty} \Big( \dfrac{p}{d}\Big) (t)=0, \ \liminf\limits_{t\to\infty} \Big( p-\dfrac{|b|^2}{d}\Big)(t)>0, \\[2.5mm]
  (-\infty,s] 
  & \textnormal{if } \,\liminf\limits_{t\to\infty} \Big( \dfrac{p}{d}\Big) (t)=0, \ \limsup\limits_{t\to\infty} \Big( p-\dfrac{|b|^2}{d}\Big)(t)<0.
  \end{cases}
\end{align*}

%
%
\item[{\rm iii)}] If $\,\dinf=-\infty$, then 
\[
\sessreg(\cA)=(-\infty,\delta_+].
\]  
\end{enumerate}
\end{proposition}

\begin{proof}
i) Since $\,\dinf \in \RR$, the function $d$ is bounded. Because $p>0$ by (A), the estimate $\Delta(t) \le d(t)$, $t\in [0,\infty)$, shows that in this case $\sessreg(\cA)$ is always bounded from above.
 
If $\inf_{t\ge 0} p(t)>0$, then $\frac 1p$ is bounded. Since $d$ is bounded, so is $b$ by Assumption (B2). Altogether this implies that $\Delta$ is bounded.


If $\inf_{t\ge 0} p(t)=0$, then there exist a sequence $(t_n)_0^\infty$ with $t_n\to\infty$ and $p(t_n)\to 0$ for $n\to\infty$. If $(\Delta(t_n))_0^\infty$ were bounded from below, this would imply $b(t_n)\to 0$ for $n\to\infty$ and hence 
\[
 \tpi(t_n,\la)=p(t_n)-\dfrac{|b(t_n)|^2}{d(t_n)-\la}\to 0, \quad n \to \infty,
\] 
for all $\la\in \big( \max\{\de_+,d_\infty\},\infty\big) \subset \RR \setminus \overline{\Delta([0,\infty))}$, 
a contradiction to Assumption (B3b). Therefore the range of $\Delta$ is not bounded from below in this case.

ii) If $\dinf=+\infty$, then $d$ is bounded from below and there exists a $t_0 \in [0,\infty)$ such that $d(t)>1$, $t\in[t_0,\infty)$.
Then Assumption (B2) implies that $\frac {|b(t)|}{|d(t)|} \le 2\beta$, $t\in [t_0,\infty)$.
By Assump\-tion (B3a) and Remark \ref{B3a}, we may fix an arbitrary $\la < \min \{ d(t): t\in [0,\infty) \}$ such that 
$\tpi(\cdot,\la)$ is bounded on $[0,\infty)$. In the following we use the identity
\begin{equation}
\label{pila}
  \tpi(t,\la) = p(t) - \frac{|b(t)|^2}{d(t)} - \la  \frac{|b(t)|^2}{d(t)^2} \frac 1 {1- \frac {\la}{d(t)}}, \quad t\in [0,\infty).
\end{equation}
 
As $\frac {|b|}d$ is bounded on $[t_0,\infty)$ and $\dinf=+\infty$, the last term on the right hand side of \eqref{pila} is bounded for $t\in[t_0,\infty)$. 
Since $\tpi(t,\la)$ is bounded for $t\in [0,\infty)$ by Assumption (B3a), $p(t) - \frac{|b(t)|^2}{d(t)}$ is bounded on $[t_0,\infty)$.
Therefore, if $\,\liminf_{t\to\infty} \big( \frac{p}{d}\big)(t)  >0$, the relation
\begin{equation}
\label{sun}  
  \Delta(t) = \frac{d(t)}{p(t)} \Big( p(t) - \frac{|b(t)|^2}{d(t)} \Big), \quad t\in[0,\infty),
\end{equation}
shows that $\Delta$ is bounded, i.e.\ $\de_\pm \in \RR$ and $\overline{\Delta([0,\infty))} = [\de_-,\de_+]$.

\smallskip 

Now suppose that $\sess^r(\cA) \!=\! \overline{\Delta([0,\infty))} \!\ne \!(-\infty,\infty)$ and $\liminf_{t\to\infty}\! \big( \frac{p}{d}\big)(t)=0$. Then there exists a sequence $(t_n)_0^\infty$ with $t_n\to\infty$ and $\frac{p(t_n)}{d(t_n)}\searrow 0$ for $n\to\infty$. 
By Assumptions (B3a), (B3b), we may choose an arbitrary $\la \notin \overline{\Delta([0,\infty))}$ and corresponding $t_\la \in [0,\infty)$ such that $\la \notin \overline{d([t_\la,\infty))}$ and the functions $\tpi(\cdot,\la)$, $\frac 1{\tpi(\cdot,\la)}$ are bounded on~$[t_\la,\infty)$.
This, together with the relation
\begin{align}
\label{p1}
  \frac{|b(t)|^2}{d(t)^2} = \Big( 1 - \frac \la {d(t)}\Big) \Big( \frac{p(t)}{d(t)} - \frac{\tpi(t,\la)}{d(t)} \Big), \quad t\in[t_\la,\infty),
\end{align}
and the assumption that $\dinf=+\infty$, implies that $\frac{|b(t_n)|^2}{d(t_n)^2}\to 0$ for $n\to\infty$. 
From \eqref{pila} and the boundedness of $\frac 1{\tpi(\cdot,\la)}$ on $[t_\la,\infty)$, 
we conclude that there exist $c>0$ and $N\in\N$ such that
\[
  \Big| p(t_n) -  \frac{|b(t_n)|^2}{d(t_n)}\Big| \ge c >0, \quad n \ge N,
\vspace{-2mm}  
\]
and hence
\[
  | \Delta(t_n) | = \Big| \frac{d(t_n)}{p(t_n)} \Big| \Big| p(t_n) -  \frac{|b(t_n)|^2}{d(t_n)} \Big| \to \infty, \quad n\to \infty.
\]
Depending on the sign of $p-\frac{|b|^2}{d}$, the range of $\Delta$ is thus either unbounded from above or from below.
Since $\Delta([0,\infty))$ is an interval and we had assumed $\ov{\Delta([0,\infty))} \ne (-\infty,\infty)$, the second claim in ii) follows.

\smallskip

iii) If $\,\dinf=-\infty$, then $d$ is bounded from above and the claim follows from the estimates  
\[
 \Delta(t) \le d(t) \le \sup_{\tau\in[0,\infty)} d(\tau) < \infty, \quad t\in [0,\infty). \qedhere
\] 
\end{proof}

%

\medskip

\noindent
{\bf 5.2. The singular part.} 
The following Assumptions (C1), (C2), and (C3) imply Assumptions (B3a), (B3b), and (C); 
we continue to use Notation \ref{not-inf} to denote limits for $t\to\infty$. 

\smallskip

Since $\sess(\cA) = \RR$ if $\Delta([0,\infty))=\RR$ by Proposition \ref{regpart}, we may assume $\Delta([0,\infty))\ne\RR$ in this subsection.

\smallskip

\noindent
\textbf{Special case of Assumption (C).} \
(C1) \ For $\la\in \RR \setminus \big( \overline{\Delta([0,\infty))} \cup \{\dinf\} \big)$ the \hspace*{7.5mm} following limits exist and are finite:
\begin{align*}
  \tpi(\cdot,\la)\linf, 
   \qquad 
   \Big( \dfrac{\partial}{\partial t}\tpi(\cdot,\la) \Big)\blinf, 
\end{align*}
\begin{enumerate}
\item[]
and $\tpi(\cdot,\la)\linf \ne 0$ for some (and hence all) $\la$. 
\item[(C2)] For $\la\in \RR \setminus \{\dinf\}$ the  limits 
$ \bigg( \dfrac{\overline{b}c}{d-\la}  \bigg)\blinf $, $ \bigg( \dfrac{\partial}{\partial t} \dfrac{\overline{b}c}{d-\la} \bigg)\blinf $ 
exist and are finite.
\item[(C3)] For $\la\in \RR \setminus \{\dinf\}$ the limit 
$ \Big( q-\la -\dfrac{|c|^2}{d-\la} \Big)\blinf$ exists and is finite.
\end{enumerate}
\medskip

\begin{remark}
\label{rem:last}
i) \ Assumptions (C1) and (C2) imply that 
\begin{align}
\label{easy}
\Big( \frac{\partial}{\partial t}\tpi(\cdot,\la) \Big)\linf=0, \quad
\Big( \frac{\partial}{\partial t} \frac{\overline{b}c}{d-\la} \Big)\linf =0.
\end{align}
Together with the definition of $\trho(\cdot,\la)$ in \eqref{pirho} and of $\tka(\cdot,\la)$ in \eqref{kappa}, 
it follows that the limits $\trho(\cdot,\la)\linf$, $\tka(\cdot,\la)\linf$ exist, are finite and have the form
\begin{align*}
  \trho(\cdot,\la)\linf = -  \Big( \frac{ 2 \Im \big( b \overline{c} \big) }{d-\la} \Big)\blinf, \quad
  \tka(\cdot,\la)\linf = \Big( q-\la -\dfrac{|c|^2}{d-\la} \Big)\blinf.
\end{align*}
ii) \ Due to \eqref{pidelta}, for the last condition in (C1) we have the equivalence
\[
  \tpi(\cdot,\la)\linf \ne 0  \iff
  \begin{cases}
  \ \liminf\limits_{t\to\infty} \Big( \dfrac{p}{|d|}\Big)(t) > 0 & \ \text{if } \dinf \in \{\pm \infty\},\\[2mm] 
  \ \liminf\limits_{t\to\infty}   p(t) > 0 & \ \text{if } \dinf \in \RR.
  \end{cases}
\]
\end{remark}

In the sequel we determine the form of the limits $\tpi(\cdot,\la)\linf$, $\trho(\cdot,\la)\linf$,  $\tka(\cdot,\la)\linf$ as functions of $\la$. They 
depend on  whether the limit $\dinf\!=\lim_{t\to\infty}d(t) \!\in\! \RR\cup\{\pm\infty\}$ (which exists by Assumption (B1)) is finite or not.

\medskip

Assumption (C1) (for two different values of $\la$) together with \eqref{pila} and  \eqref{pila0} implies that 
the limits 
\begin{align*}
\left\{ \begin{array}{rrl}
 \Big( p - \dfrac{|b|^2}{d} \Big)\linf, \ \ & \Big( \dfrac{|b|^2}{d^2} \Big)\blinf \ \  &  \text{ if } \dinf=\pm \infty, \\[4mm]
 p\linf \ (\ge 0)\,, \ \  & (|b|^2)\linf \ \  & \text{ if } \dinf\in\RR,
\end{array} \right.
\end{align*}
exist, are finite and that the limit $\tpi(\cdot,\la)\linf$ has the form 
\begin{align}
\label{tpi-inf}
 \tpi(\cdot,\la)\linf = \Big( p-\frac{|b|^2}{d-\la} \Big)\blinf \! = \!\begin{cases}
                     \ \Big( p - \dfrac{|b|^2}{d} \Big)\linf - \la \Big( \dfrac{|b|^2}{d^2} \Big)\blinf & \text{ if } \dinf=\pm\infty, \\[3mm]
                     \ \hspace{2mm} p\linf -\dfrac{(|b|^2)\linf}{\dinf-\la}  & \text{ if } \dinf\in\RR.
                    \end{cases} \hspace{-4mm}
\end{align}
%
Remark \ref{rem:last} implies that the limits
\begin{align*}
\left\{ \begin{array}{rl}
 \Big( \dfrac{\Im ( b \overline{c})}{d} \Big)\linf \ \  &  \text{ if } \dinf=\pm \infty, \\[4mm]
 ( \Im ( b \overline{c}))\linf \ \  & \text{ if } \dinf\in\RR,
\end{array} \right.
\end{align*}
exist, are finite and that $\trho(\cdot,\la)\linf$ has the form 
\begin{align}
\label{trho-inf}
 \trho(\cdot,\la)\linf = \begin{cases}
                     \ -2 \Big( \dfrac{\Im ( b \overline{c})}{d} \Big)\linf  & \text{ if } \dinf=\pm \infty, \\[3mm]
                     \ -2 \dfrac{( \Im ( b \overline{c}))\linf}{\dinf-\la}  & \text{ if } \dinf\in\RR.
                    \end{cases}
\end{align}
%
%
Assumption (C3) implies that the limits 
\begin{align*}
\left\{ \begin{array}{rrl}
 \Big( q - \dfrac{|c|^2}{d} \Big)\linf, \ \ & \bigg( \dfrac{|c|^2}{d^2} \bigg)\blinf \ \  &  \text{ if } \dinf=\pm \infty, \\[4mm]
 q\linf , \ \  & (|c|^2)\linf \ \  & \text{ if } \dinf\in\RR,
\end{array} \right.
\end{align*}
exist, are finite and that $\tka(\cdot,\la)\linf$ has the form 
\begin{align}
\label{tka-inf}
 \tka(\cdot,\la)\linf = \begin{cases}
                     \ \Big( q - \dfrac{|c|^2}{d} \Big)\linf  - \la \Big( 1+ \Big( \dfrac{|c|^2}{d^2} \Big)\blinf\, \Big) 
                     & \text{ if } \dinf=\pm\infty, \\[3mm]
                     \ q\linf -\la - \dfrac{(|c|^2)\linf}{\dinf-\la}  & \text{ if } \dinf\in\RR.
                    \end{cases}
\end{align}


\begin{proposition}
\label{struc-esssing}
Suppose that Assumptions {\rm (A)}, {\rm (B1)}, {\rm (B2)}, {\rm (C1)}, {\rm (C2)}, and {\rm (C3)} are satisfied. 
Then 
there exist $s_-$, $s_+$, $s\in\RR$, $s_- \le s_+ \le s$, such that the following hold:
\begin{enumerate}
\item[{\rm i)}] 
if $\dinf \in \RR$, then 
\end{enumerate}
\[
   \sesssing\bigl(\cA\bigr) = 
   \begin{cases}
   \ds \hspace{3mm}
   \left. \begin{array}{r}
   \big( [s_-,s_+] \cup [s,+\infty) \big) \setminus \big( \ov{\Delta([0,\infty))} \cup \{ d_\infty\} \big) \\ 
   \text { or } \ [s,+\infty) \ \setminus \big( \ov{\Delta([0,\infty))} \cup  \{ d_\infty\} \big)
   \end{array} \!\!\right\}
   & \text{if } \ p\linf \!>\! 0, \\[3mm] 
   \big( (-\infty,s_-] \cup [s_+,+\infty) \big) \setminus \big( \ov{\Delta([0,\infty))} \cup  \{ d_\infty\} \big)
   & \text{if } \ p\linf \!=\! 0;
   \end{cases}
\]
\begin{enumerate}
\item[{\rm ii)}]
if $\dinf = + \infty$, then 
\[
  \sesssing\bigl(\cA\bigr) = 
  \begin{cases}
  [s_-,s_+] \setminus \ov{\Delta([0,\infty))}  
  & \text{if } \ \Big( \dfrac{|b|^2}{d^2} \Big)\blinf > 0, \\[2.5mm]
  [s,+\infty) \setminus \ov{\Delta([0,\infty))}
  & \text{if } \ \Big( \dfrac{|b|^2}{d^2} \Big)\blinf = 0, \, \Big( p - \dfrac{|b|^2}{d} \Big)\blinf>0, \\[2.5mm]
  (-\infty,s] \setminus \ov{\Delta([0,\infty))}
  & \text{if } \ \Big( \dfrac{|b|^2}{d^2} \Big)\blinf = 0, \, \Big( p - \dfrac{|b|^2}{d} \Big)\blinf<0;
  \end{cases}
\]
\item[{\rm iii)}]
if $\dinf = -\infty$, then 
\[
  \sesssing\bigl(\cA\bigr) = 
  \begin{cases}
  [s_-,s_+] \setminus \ov{\Delta([0,\infty))}
  & \text{if } \ \Big( \dfrac{|b|^2}{d^2} \Big)\blinf > 0, \\[2.5mm]
  [s,+\infty) \setminus \ov{\Delta([0,\infty))}
  & \text{if } \ \Big( \dfrac{|b|^2}{d^2} \Big)\blinf = 0.
  \end{cases}
\]
\end{enumerate}
\end{proposition}

\begin{proof}
By Assumptions {\rm (C1)}, {\rm (C2)}, and {\rm (C3)}, the condition $\Dis(\la) \ge 0$ in \eqref{def-discr} for $\la$ 
to belong to $\sesssing\bigl(\cA\bigr)$ is equivalent to 
\begin{align}
\label{sess-formula2}
   \trho(\cdot,\la)\linf^2- 4\tka(\cdot,\la)\linf \tpi(\cdot,\la)\linf \ge 0.
\end{align}

i) If $\dinf \in \RR$, then 
\eqref{sess-formula2} has the form
\begin{equation}
\label{veryverylast}
( \Im ( b \overline{c}))\linf ^2 \ge   \big( (q\linf -\la)(\dinf-\la) - (|c|^2)\linf \big) 
 \big( p\linf (\dinf-\la) - (|b|^2)\linf \big). 
\end{equation}
The polynomial on the right hand side is at most cubic in $\la$ with leading coeffcient $-p\linf$; if $p\linf=0$, then the leading quadratic coefficient $-(|b|^2)\linf$ is negative; note that $(|b|^2)\linf\ne 0$ since otherwise $\frac 1{\pi(\cdot,\la)}$ would not be bounded.

ii), iii) If $\dinf = \pm \infty$, then the relations \eqref{tpi-inf}, \eqref{trho-inf}, and \eqref{tka-inf} imply that \eqref{sess-formula2} has the form
\[
 \Big( \dfrac{\Im ( b \overline{c})}{d} \Big)\linf ^2 
 \ge  
 \Big( \Big( p - \dfrac{|b|^2}{d} \Big)\blinf - \la \Big( \dfrac{|b|^2}{d^2} \Big)\blinf \Big) 
 \Big( \Big( p - \dfrac{|b|^2}{d} \Big)\blinf  - \la \Big( 1+ \Big( \dfrac{|b|^2}{d^2} \Big)\blinf\, \Big) \Big).
\]
The polynomial on the right hand side is at most quadratic in $\la$; 
more precisely, if $\big( \frac{|b|^2}{d^2} \big)\blinf > 0$ it is quadratic with non-negative leading coefficient and non-negative discriminant; 
if $\big( \frac{|b|^2}{d^2} \big)\blinf = 0$, it is linear with leading coefficient having the opposite sign as $\big( p - \frac{|b|^2}{d} \big)\blinf$, and constant $0$ if the latter is $0$. Moreover, if $\dinf = -\infty$, then $\big( p - \frac{|b|^2}{d} \big)\blinf >0$.

\smallskip

Now all claims in i), ii), iii) follow from elementary sign considerations.
\end{proof}

\begin{remark}
\label{rem-p1}
i) \ If $\dinf \in \RR$, the last factor in \eqref{veryverylast} is equal to $p_\infty(\Delta_\infty-\la)$ by \eqref{Delta} and hence $\Delta_\infty\in [\de_-,\de_+]$ satisfies inequality \eqref{sess-formula2}.
\\[0.5mm]
ii) \ If $\dinf = \pm \infty$, the relation \eqref{p1} shows that $\Big( \dfrac{|b|^2}{d^2} \Big)\blinf = \Big( \dfrac{p}{d} \Big)\blinf $.
\end{remark}

\medskip

\noindent
{\bf 5.3. The whole essential spectrum.} 
If we combine the information about the regular and singular part, we obtain the following result for the essential spectrum of $\cA$.

\begin{theorem}
\label{unified}
Suppose that Assumptions {\rm (A)}, {\rm (B1)}, {\rm (B2)}, {\rm (C1)}, {\rm (C2)}, and {\rm (C3)} are satisfied. 
Then the essential spectrum of every closed symmetric extension $\cA$ of the operator $\cA_0$ in \eqref{A0}  has the following form:
\begin{enumerate}
\item[{\rm i)}]
if $d_{\infty}\in \RR$, \vspace{-3mm} then
\end{enumerate}
\begin{equation*}
\sess(\cA) \setminus \{\dinf\} 
= \begin{cases} 
\ds \hspace{-2mm}
\left. \begin{array}{r}
\big( [\min\{s_-,\delta_-\}, \max\{s_+,\delta_+\}] \cup [s,\infty) \big) \setminus \{d_\infty\}\!\!\\
\text { or } \ \big( [ \min\{s,\delta_-\},\infty) \big) \setminus \{d_\infty\}\!\!
\end{array} \right\}
&\textnormal{if } p\linf\!>0, \\ 
\hspace{10mm}
\ds  \big( (-\infty, \max\{s_-,\delta_+\}] \cup [s_+, \infty) \big) \setminus \{d_\infty\} & \textnormal{if } p\linf\!=0; \\
\end{cases}
\end{equation*}
\begin{enumerate}
\item[{\rm ii)}]
if $d_{\infty}=+\infty$, then 
\begin{equation*}
\sess(\cA)= \begin{cases} 
\ds [\min\{s_-,\delta_-\}, \max\{s_+,\delta_+\}]
& \textnormal{if } \Big( \dfrac{p}{d} \Big)\blinf\!\!>0, \\[2.5mm] 
\ds [\min\{s,\delta_-\}, \infty) & \textnormal{if } \Big( \dfrac{p}{d} \Big)\blinf\!\!=0, \, \Big( p - \dfrac{|b|^2}{d} \Big)\blinf\!\!>0,\\[2.5mm]
\ds (-\infty, \max\{s,\delta_+\}] & \textnormal{if } \Big( \dfrac{p}{d} \Big)\blinf\!\!=0, \, \Big( p - \dfrac{|b|^2}{d} \Big)\blinf\!\!<0; 
\end{cases}
\end{equation*}
\item[{\rm iii)}]
if $d_{\infty}=-\infty$, then 
\begin{equation*}
\sess(\cA)= 
\begin{cases}  \ds (-\infty, \max\{s_+,\delta_+\}] 
& \textnormal{if } \Big( \dfrac{p}{d}  \Big)\blinf\!\!>0,\\[2.5mm]
\ds (-\infty, \delta_+] \cup [s,\infty) & \textnormal{if } \Big( \dfrac{p}{d}  \Big)\blinf\!\!=0; 
\end{cases}
\end{equation*}
\end{enumerate}
here $\de_-$, $\de_+$ are given by \eqref{depm} and $s_-$, $s_+$, $s$ are as in Proposition {\rm \ref{struc-esssing}}.
\end{theorem}

\begin{proof}
All claims follow from the respective claims in Propositions \ref{prop2:regpart} and \ref{struc-esssing}, observing Remark \ref{rem-p1} i) and ii).
%
\end{proof}

Note that the intervals in i) and iii) of Theorem \ref{unified} need not be disjoint; 
e.g.\ if $d_\infty\in\RR$ and $p_\infty>0$ it may happen that between the two intervals $[s_-,s_+]$, $[s,\infty)$ in $\sesssing(\cA)$ there is a gap which is covered by $\sessreg(\cA)\!=\!\ov{\Delta([0,\infty))}\!=\![\de_-,\de_+]$.


\medskip

An even more particular case of Theorem \ref{ess_spec1} is that
all coefficients $p$, $q$, $b$, $c$, and $d$ have  limits at $\infty$, or that they are even constant.

\begin{example}
\label{coeff_lim}
Suppose that all the limits 
\[
 p\linf, \quad q\linf, \quad b\linf, \quad c\linf, \quad \dinf \quad \textnormal{and} \quad 
 (p')\linf, \quad  (b')\linf, \quad (c')\linf, \quad (d')\linf, \quad 
\]
exist and are finite $($which implies that the limits of all derivatives are $0$), that $p\linf>0$, and that $\Im(\ov{b\linf}c\linf)=0$
$($which holds e.g.\ if $b$ and $c$ are real-valued$)$.
Then, \vspace{-3mm} with
\begin{align*}
  \Delta\linf 
  \defequ \dinf-\frac{|b\linf|^2}{p\linf}, \quad \de_-=\inf_{t\in[0,\infty)} \!\Delta(t), \quad \de_+=\sup_{t\in[0,\infty)} \!\Delta(t),\\
  \Lambda^\pm\linf \defequ \frac{q\linf+\dinf}2 \pm \sqrt{\left(\frac{q\linf-\dinf}2\right)^2+|c\linf|^2},
\end{align*}
the essential spectrum of every closed symmetric extension $\cA$ of $\cA_0$ in \eqref{A0} is given by
\[
  \sess(\cA) \setminus\{\dinf\} \!=\! \Bigl( [\de_-,\de_+] \,\cup\, \bigl[\min\bigl\{\Delta\linf\!, \Lambda^-\linf\bigr\},
  \max\bigl\{ \Delta\linf\!, \Lambda^-\linf\bigr\}\bigr]
  \,\cup\, \bigl[\Lambda^+\linf\!, \infty\bigr)\Bigr)\setminus\{\dinf\};
\]
note that the points $\Lambda^\pm\linf$ are the eigenvalues of the $2\times 2$ matrix
\[
  \begin{pmatrix} q\linf & \ov{c\linf} \\[1ex] c\linf & \dinf \end{pmatrix}
\]
whose entries are the limits of the lowest order terms in $\cA$.
%
%
In fact, the assumptions imply that not only Assumptions (B) and (C) are satisfied, but even the stronger assumptions (C1) and (C2) 
of Section~3.2 with $(\rho(\cdot,\la))\linf=0$. Since $\dinf \in \RR$, we know from Theorem \ref{unified} that
$\sess(\cA)$ consists of at most two intervals, possibly one bounded and one unbounded interval extending to $+\infty$, plus perhaps the point $\dinf$.

More precisely, Proposition~\ref{prop2:regpart}~i) and Theorem~\ref{ess_spec1} imply that $\sessreg(\cA)=\ov{\Delta([0,\infty))} = [\de_-,\de_+]$ and
\begin{align*}
  \sesssing(\cA) & = \,
  \bigr\{\la\in\RR\setminus\big( \ov{\Delta([0,\infty))} \cup \{\dinf\} \big) \colon 
  \Big( \dfrac{\tka(\cdot,\la)}{\tpi(\cdot,\la)} \Big)\blinf \le 0\bigr\} \\
  &=\Bigr\{\la\in\RR\setminus \big( \ov{\Delta([0,\infty))} \cup \{\dinf\} \big) \colon 
  \frac{(q\linf-\la)(\dinf-\la)-|c\linf|^2}{p\linf(\dinf-\la)-|b\linf|^2} \le 0\Bigr\}.
\end{align*}
The zeros of the quadratic polynomial in the numerator above are $\Lambda^-\linf \le \Lambda^+\linf$, 
while the zero of the linear function in the denominator is $\Delta\linf$.
The claim for $\sesssing\bigl(\cA\bigr)$ follows from mere sign considerations if we observe, in addition, that
\[
  \Lambda^+\linf
  \ge \frac{q\linf+\dinf}2 + \left|\frac{q\linf-\dinf}2\right|
  = \max\{q\linf, \dinf\} \ge \dinf
  \ge \dinf-\frac{|b\linf|^2}{p\linf} = \Delta\linf,
\]
where we have used that $p\linf>0$.

A particular case of the example above is the block operator matrix
\[
  \cA_0= \matrix{cc}{- \ds p\linf \frac{\d^2}{\d t^2} + q\linf & -\ov{b\linf} \dt \\[1.5ex] b\linf \dt & \dinf}
\]
with constant coefficients, $c\equiv 0$, and $p\linf >0$ in $L^2(0,\infty)\oplus L^2(0,\infty)$. 
\vspace{-1mm}Then
\[
  \Delta\linf = \de_- = \de_+ = \dinf-\frac{|b\linf|^2}{p\linf}, \quad
  \Lambda^+\linf = \max\{ \dinf, q\linf \}, \quad \Lambda^-\linf = \min\{ \dinf, q\linf \}
\vspace{-1mm}
\]
and \vspace{-1mm}
\[
  \sess(\cA) \!=\! \begin{cases}
  \bigl[\dinf\!-\!\frac{|b\linf|^2}{p\linf}\!, \dinf \bigr] \cup \bigl[ q\linf, \infty\bigr)
  \!\!&\text{if } \dinf \!\le\! q\linf, \\[1.5ex]
  \bigl[\min\bigl\{\dinf\!-\!\frac{|b\linf|^2}{p\linf}\!, \alpha \bigr\},
  \max\bigl\{\dinf\!-\!\frac{|b\linf|^2}{p\linf}\!, q\linf \bigr\}\bigr] \cup \bigl[\dinf,\infty\bigr)
  \!\!&\text{if } \dinf \!>\! q\linf.
  \end{cases}
\]
%
%
Hence, if $q\linf=\dinf$, then $\sess(\cA)$ consists of one
interval; otherwise, it consists of two intervals.
\end{example}


%
%

\section{Matrix differential operators on $(0,1]$}
\label{sec4}

In this section we consider matrix differential operators defined on $(0,1]$ for which $0$ is a singular end-point and which are symmetric 
in a product of weighted $L^{2}$-spaces.

Using a suitable transformation to the interval $[0,\infty)$, we establish assumptions on the behaviour of the original coefficients in $(0,1]$ allowing us to prove an analogue of Theorem \ref{ess_spec1}.

We consider the differential \vspace{-1mm} expressions
\begin{alignat*}{2}
  \tau_{\tA}&:= - \frac 1{w_1} \dx \tp \dx w_2 + \tq, & \qquad
  \tau_{\tB}&:= - \dx \ov{\,\tb\,} - \ds \frac {w'}w \ov{\,\tb\,}  + \ov{\,\tc\,}, \\
  \tau_{\tC}&:= \tb \dx  + \tc, &\qquad
  \tau_{\tD}&:= \td,
\end{alignat*}
with coefficient functions $\tp$, $\tq$, $\tb$, $\tc$, $\td$, and $w_1$, $w_2$, $w:=w_1w_2$ satisfying the following. 


\vspace{2mm}

\textbf{Assumption ($\widetilde{\mbox{A}}$).} \ $\tp$, $\td\in C^{2}((0,1],\RR)$, $\tb \in C^2((0,1],\CC)$, $\tc\in C^1((0,1],\CC)$, $\tq\in C((0,1],\RR)$ with $\tp(x) > 0$, $x\in (0,1]$, and $w_1,w_2\in C^2((0,1],\RR)$ with $w=w_1w_2 >0$ on $(0,1)$.
\vspace{2mm}
\newline
We denote by $\tA_0, \tB_0, \tC_0$, and $\tD_0$ the operators in the weighted Hilbert space $L^2((0,1),w)$  induced by the differential expressions $\tau_{\tA}$,
$\tau_{\tB}$, $\tau_{\tC}$, $\tau_{\tD}$ with domains 
\begin{equation*}
\cD(\tA_0):=C^{2}_0((0,1)), \quad \cD(\tB_0)=\cD(\tC_0):=C^{1}_0((0,1)), \quad \cD(\tD_0):=C_0((0,1)).
\end{equation*}
In the Hilbert space $L^2((0,1),w)^2=L^2((0,1),w)\oplus L^2((0,1),w)$ we consider the matrix differential operator 
\begin{equation}
\label{matrix01}
\begin{aligned}
  & \wt\cA_0 \defeq \matrix{cc}{\tA_0 & \tB_0 \\ \tC_0 & \tD_0} =
  \matrix{cc}{- \displaystyle{\frac 1{w_1}} \dx \tp \dx w_2 + \tq
  & - \dx \ov{\,\tb\,} - \ds \frac {w'}w \ov{\,\tb\,} + \ov{\,\tc\,} \\[2ex]
              \tb \dx  + \tc & \td},\\
 & \cD(\wt\cA_0) := C_0^{2}((0,1)) \oplus C_0^{1}((0,1)).
\end{aligned}
\end{equation}
Note that, with respect to the scalar product in $L^2((0,1),w)$, the differential expression defining $\tA_0$ is symmetric and those defining $\tB_0$, $\tC_0$ are formally adjoint to each other. 

Consider the first Schur complement of the matrix $\wt\cA_0$, which is formally given by the second order differential expression
\begin{align*}
\tS(\lambda) & :=\tau_{\tA}-\lambda-\tau_{\tB}(\tau_{\tD}-\lambda)^{-1}\tau_{\tC}
              =-\ttpi(\cdot,\lambda) \ds \frac{\d^2}{\d x^2}+\ttrho(\cdot,\lambda)\,\ii\ds\frac{\d }{\d x}+\tkappa(\cdot,\lambda)   
\end{align*}
for $\la \in \CC \setminus \wt d((0,1])$ with coefficients
\begin{align}
\label{tpirho}
\ttpi(\cdot,\lambda) &:= \ds \frac{w_2}{w_1}\tp-\ds \frac{|\tb|^2}{\td-\lambda}, \hspace{4mm}
\ttrho(\cdot,\lambda):=-\ds \frac{2\Imag(\tb\hspace{0.3mm} \ov{\,\tc\,})}{\td-\lambda}+
\ds \mbox{i}\frac{1}{w}\frac{\D}{\D x}\bigl(w\ttpi(\cdot,\lambda)\bigr),\\
\label{tka}
\tkappa(\cdot,\lambda) &:= \tq-\lambda-\ds \frac{|\tc|^2}{\td-\lambda}+\ds \frac{1}{w}\Bigl(w\frac{\ov {\,\tb\,}\tc}{\td-\lambda}\Bigr)'-\ds \frac{(\tp w_2')'}{w_1}.
\end{align}
Define 
\begin{equation}
\label{tDelta}
\tDelta(x):=\td(x)-\ds \frac{w_1(x)}{w_2(x)}\frac{|\tb(x)|^2}{\tp(x)}, \hspace{3mm} x\in (0,1].
\end{equation}\vspace{2mm}\newline
\textbf{Assumption ($\widetilde{\mbox{B}}$).} \ ($\widetilde{\mbox{B}}1$) The possibly improper limit $\td_0:=\ds \lim_{x\to 0+}\td(x)$ exists;\newline
($\widetilde{\mbox{B}}2$) there exist constants $\tbeta$, $\tgamma>0$ such that 
\begin{equation*}
\Bigl|\frac{\tb(x)}{x}\Bigr|\leq \tbeta(|\td(x)|+1), \hspace{2mm} |\tc(x)|\leq \tgamma(|\td(x)|+1) \hspace{3mm} \mbox{for}\hspace{1mm} \mbox{all}\hspace{1mm} x\in (0,1];
\end{equation*}
($\widetilde{\mbox{B}}3$a) for some (and hence all) $\lambda\in \RR\setminus \{\td_0\}$ there exists an $x_{\lambda}\in (0,1)$ such that 
\begin{equation*}
\frac{\ttpi(x,\lambda)}{x^2} \hspace{3mm} \mbox{is a bounded function (of $x$) on}\hspace{2mm} (0,x_{\lambda}];
\end{equation*}
\newline\noindent
$(\widetilde{\mbox{B}}3$b) for all $\lambda\in \RR\setminus \bigl(\hspace{0.1mm} \ov{\,\tDelta((0,1])\,}\cup \{\td_0\}\bigr)$ there exists an $x_{\lambda}\in (0,1)$ such that 
\begin{equation*}
\ds \frac{x^2}{\ttpi(x,\lambda)},\hspace{1mm} \ds \frac{\ttrho(x,\lambda)}{x},\hspace{1mm} \tkappa(x,\lambda)\hspace{3mm} \mbox{are bounded functions (of $x$) on} \hspace{2mm} (0,x_{\lambda}].
\end{equation*}
\newline
\textbf{Assumption ($\widetilde{\mbox{C}}$).}  ($\widetilde{\mbox{C}}1$) 
For $\lambda\!\in\!\RR\!\setminus\!\bigl(\hspace{0.1mm} \ov{\,\tDelta((0,1])\,}\cup \{\td_0\}\bigr)$ the following limits exist 
and are~finite:
\begin{equation}
\label{mainass-t}
\ttrho_0(\lambda) \defeq \lim_{x\to 0+}\frac{x\ttrho(x,\lambda)}{\ttpi(x,\lambda)}, \hspace{3mm} \tkappa_0(\lambda) \defeq \lim_{x\to 0+}\frac{x^2\tkappa(x,\lambda)}{\ttpi(x,\lambda)}; 
\end{equation}
\newline
($\widetilde{\mbox{C}}2$) \ the limit 
$\ds \lim_{x\to 0+} \frac{x^2w''(x)}{w(x)}$ exists and is finite.

\begin{remark}
Note that, as in Remark \ref{B3a}, it is sufficient to require in ($\widetilde{\mbox{B}}3$a) that the boundedness of the function $\frac{\ttpi(x,\lambda)}{x^2}$ holds for some $\la_0 \in \RR\setminus \{\td_0\}$ because, for any other $\la \in \RR\setminus \{\td_0\}$,
\begin{equation*}
\ds \frac{\tpi(x,\lambda)}{x^2}-\frac{\tpi(x,\lambda_0)}{x^2}=\frac{\lambda_0-\lambda}{\bigl(\td(x)-\lambda_0\bigr)\bigl(\td(x)-\lambda\bigr)}\frac{|\tb(x)|^2}{x^2}, \quad x \in [\max\{x_\la,x_{\la_0}\},1].
\end{equation*}
\end{remark}

\begin{remark}
Assumptions ($\widetilde{\mbox{B}}3$a) and ($\widetilde{\mbox{C}}1$)  imply ($\widetilde{\mbox{B}}3$b) (compare Remark \ref{last!}).
\end{remark}


In the following transformation of the matrix differential operator $\wt \cA_0$ in \eqref{matrix01}, the function 
\begin{equation}
\label{W}
  W(t):= 1 + \e^{-t} \frac{w'(\e^{-t})}{w(\e^{-t})}, \quad t\in [0,\infty),
\end{equation}
plays a role. Our assumptions guarantee that $W$ and $W'$ have limits at $\infty$; more precisely:

\begin{lemma}
\label{rem-orif0}
If Assumptions {\rm ($\widetilde{\mbox{A}}$)}, {\rm ($\widetilde{\mbox{B}}$)}, and {\rm ($\widetilde{\mbox{C}}$)} hold, then
\begin{align}
\label{gronwall}
\lim_{x\to 0+} x \frac{\frac{\D}{\D x}\ttpi(x,\lambda)}{\ttpi(x,\lambda)} = 2
\end{align}
and hence
\begin{align}
\label{remw0}
&\lim_{t\to\infty} W(t) \mbox{ exists and is finite}, \quad \lim_{t \to \infty} W'(t) = 0, \\
\label{remw0a}
&\ds \Imag(\ttrho_0(\lambda))
= 1 + \lim_{t\to\infty} W(t).
\end{align}
\end{lemma}

\begin{proof}
In order to prove \eqref{gronwall}, let $y(x):= \frac{\ttpi(x,\lambda)}{x^2}$, $x \in (0,1]$. By Assumptions ($\widetilde{\mbox{B}}3$a), ($\widetilde{\mbox{B}}3$b),
$y$ and $\frac 1y$ are bounded near $0$. Similarly as in the proof of Lemma~\ref{rem-orif}, using Gronwall's lemma, one can show that this implies
\[
  0 = \lim_{x\to 0+} x \frac{y'(x)}{y(x)} = \lim_{x\to 0+} \Big( x \frac{\frac{\D}{\D x}\ttpi(x,\lambda)}{\ttpi(x,\lambda)} -2 \Big).
\]
This, together with Assumption ($\widetilde{\mbox{C}}1$) and with
\[
 \ds \Imag(\ttrho_0(\lambda)) = \lim_{x\to 0+} \Big( x \frac{\frac{\D}{\D x}\ttpi(x,\lambda)}{\ttpi(x,\lambda)}  + x \frac{w'(x)}{w(x)} \Big),
\]
shows that $W$ has a finite limit at $\infty$ and that \eqref{remw0a} holds.
Further, it is not difficult to check that Assumption ($\widetilde{\mbox{C}}2$) implies that also $W'$ has a limit at $\infty$ and hence
$\lim_{t\to\infty} W'(t)=0$.
\end{proof}

\begin{theorem}
\label{essspec01}
Suppose that Assumptions $\textnormal{($\widetilde{\textnormal{A}}$)}$, $\textnormal{($\widetilde{\textnormal{B}}$)}$, and $\textnormal{($\widetilde{\textnormal{C}}$)}$  are satisfied.
Then the essential spectrum of every closed symmetric extension $\wt\cA$ of the operator $\wt\cA_0$ in $L^2((0,1),w)^2$ is given by 
\begin{equation*}
\sess(\wt\cA) \setminus \{\td_0\} = \big( \sessreg(\wt\cA)\cup\sesssing(\wt\cA) \big) \setminus \{\td_0\},
\end{equation*}
where
\begin{align*}
\sessreg(\wt\cA)&\!:=\!\ov{\,\tDelta((0,1])\,},\\
\sesssing(\wt\cA)&\!:=\!\left\{\lambda\in\RR\setminus (\ov{\,\tDelta((0,1])\,} \!\cup\! \{\td_0\} ): 
\Re \big( (\widetilde\rho_0(\lambda)-\ii)^2 \big) -  4\Re \big( \widetilde\kappa_0(\lambda)\big) \ge 0
\right\}
\end{align*}
with $\tDelta$ given by \eqref{tDelta} and $\widetilde\rho_0(\lambda)$ and $\widetilde\kappa_0(\lambda)$ defined as in \eqref{mainass-t}.
\end{theorem}

\vspace{3mm}

The idea of the proof is to transform the space $L^2((0,1),w)$ unitarily onto the space $L^2(0,\infty)$ so that the corresponding transformed operator matrix has the form required in Theorem~\ref{ess_spec1}. To this end, we need the following lemma.

\begin{lemma}
\label{0.2}
Suppose that $w\in C^2((0,1])$, $w>0$, 
and define $\psi\in C^2([0,\infty))$~as 
\begin{equation}
\label{psi}
\psi(t):=\sqrt{\e^{-t}w(\e^{-t})}, \quad t\in [0,\infty).
\end{equation}
Then the operator 
\begin{equation*}
U:L^2((0,1),w)\to L^2(0,\infty), \quad (Uu)(t):=\psi(t)u(\e^{-t}),
\end{equation*}
is unitary and, with $W$ defined as in \eqref{W},
\begin{align*}
U\dx U^{-1}
=& -\e^t\dt - \ds \frac 12 \e^t W(t),
\\
U \ds \frac{\d^2}{\d x^2}U^{-1}
= &\, \e^{2t}\ds \frac{\d^2}{\d t^2} + \e^{2t} \big( 1 + W(t) \big) \dt  + \frac 12 \e^{2t}  \Big( W(t) + \frac 12 W(t)^2 +W'(t) \Big).
\end{align*}
\end{lemma}

\begin{proof}
That $U$ is unitary is easily seen from the relations $w(x)=\frac{1}{x}\left(\psi\left(-\log x\right)\right)^2$, $x\in (0,1]$, and
\begin{equation*}
\ds \int_0^{\infty}|(Uu)(t)|^2\d t=\ds \int_0^1|\psi(-\log x)|^2|u(x)|^2\frac{1}{x}\d x=\ds \int_0^1w(x)|u(x)|^2 \d x.
\end{equation*}
The expressions for $U\ds\dx U^{-1}$ and $U\ds \frac{\d^2}{\d x^2}U^{-1} = \Big( U\dx U^{-1} \Big)^2$ are easy to verify since the inverse of $U$ is given by
\begin{align*}
\left(U^{-1}\widetilde{u}\right)(x)&
= \frac 1{\sqrt{xw(x)}}\widetilde{u}(-\log x), \quad x\in(0,1]. 
\qedhere
\end{align*}
\end{proof}

\begin{proof}[Proof of Theorem {\rm \ref{essspec01}}.]
Consider the block operator matrix
\[
  \cA_0 \defeq \matrix{cc}{U & 0 \\ 0 & U} \matrix{cc}{\tA_0 & \tB_0 \\ \tC_0 & \tD_0}\matrix{cc}{U^{-1} & 0 \\ 0 & U^{-1}}
\]
acting in the Hilbert space $L^2(0,\infty)\oplus L^2(0,\infty)$.
Because $U$ is unitary, the essential spectra of arbitrary closed symmetric extensions $\wt\cA$ of $\wt\cA_0$ and $\cA$ of $\cA_0$ coincide,
\[
  \sess(\wt\cA) = \sess(\cA).
\]
It is not difficult to see that $\cA_0$ is of the form \eqref{A0} with coefficient functions
\begin{equation*}
\begin{array}{ll}
p(t)=\e^{2t}\ds \frac{w_2(\e^{-t})}{w_1(\e^{-t})}\tp(\e^{-t}), \quad q(t)=\tq(\e^{-t}), \ &b(t)=-\e^t\tb(\e^{-t}),\\
\vspace{2mm}
c(t)=\tc(\e^{-t})-\ds \frac{1}{2}\Bigl(\e^t+\frac{w'(\e^{-t})}{w(\e^{-t})}\Bigr)\tb(\e^{-t}), \ &d(t)=\td(\e^{-t}),
\vspace{-3mm}
\end{array}
\end{equation*}
and that, hence, the function $\Delta$ defined in \eqref{Delta} has the form
\begin{equation}
\Delta(t)=\td(\e^{-t})-\ds\frac{w_1(\e^{-t})}{w_2(\e^{-t})}\frac{|\tb(\e^{-t})|^2}{\tp(\e^{-t})}, \quad t\in [0,\infty).
\end{equation}
By Assumption ($\widetilde{\mbox{B}}1$) we have $\dinf=\lim_{t\to\infty}d(t)=\lim_{x\to 0+}\td(x)=\td_0$ so that Assumption (B1) holds. 
Assumption ($\widetilde{\mbox{B}}2$) together with Lemma~\ref{rem-orif0} guarantees that $b$ and $c$ satisfy Assumption (B2).
In order to check Assumptions (B3a), (B3b), and (C), let $\widetilde S(\lambda)$, $\la \in \RR \setminus \overline{\widetilde d((0,1])}$, be the first Schur complement of~$\widetilde \cA_0$. 
Then the first Schur complement $S(\la)$ of $\cA_0$ is given by
\begin{equation*}
\ds S(\lambda)=U\tS(\lambda)U^{-1}=-\ttpi(\e^{-t},\lambda)\,U\ds \frac{\d^2}{\d x^2}U^{-1}\!\!+\ii\,\ttrho(\e^{-t},\lambda)U\dx U^{-1}\!\!+\ds \tkappa(\e^{-t},\lambda).
\end{equation*}
Lemma \ref{0.2} with $w=w_1w_2$ yields that
\begin{equation*}
S(\lambda)=-\pi(t,\lambda)\frac{\d^2}{\d t^2}+\rho(t,\lambda)\,\mbox{i}\dt+{\kappa}(t,\lambda)
\end{equation*}
where, for $t\in[0,\infty)$,
\begin{equation}\label{pirhokappa}
\begin{array}{rl}
\pi(t,\lambda)\!=&
\hspace{-2.5mm}\ds \ \,\e^{2t}\ttpi(\e^{-t}\!,\lambda),\\[1mm]
\rho(t,\lambda)\!=&\hspace{-2.5mm}\ds -\e^t\ttrho(\e^{-t}\!,\lambda) + 
\ii\bigl(1+W(t)\bigr)\tpi(t,\lambda), \\
\kappa(t,\lambda)\!=&\hspace{-2.5mm}\ds \tkappa(\e^{-t}\!\!,\lambda)
\!-\!\frac{\ii}{2} \e^t W(t) \ttrho(\e^{-t}\!\!,\lambda) 
\ds \!-\! \frac 12 \Big( W(t) \!+\! \frac 12 W(t)^2 \!\!+\! W'(t) \!\Big)\tpi(t,\lambda). \hspace*{-4mm}
\end{array}
\end{equation}
It is easy to see that the functions given by the formulas in \eqref{pirhokappa} satisfy Assumptions (B3a), (B3b) due to Assumptions ($\widetilde{\mbox{B}}3$a), ($\widetilde{\mbox{B}}3$b), ($\widetilde{\mbox{C}}1$), ($\widetilde{\mbox{C}}2$), and Lemma~\ref{rem-orif0}. 
Moreover, Assumptions ($\widetilde{\mbox{C}1}$), ($\widetilde{\mbox{C}}2$), and Lemma~\ref{rem-orif0} imply that the~limits 
\begin{align*}
\Big( \dfrac{\trho(\cdot,\la)}{\tpi(\cdot,\la)} \Big)\blinf, \qquad \Big( \dfrac{\tka(\cdot,\la)}{\tpi(\cdot,\la)} \Big)\blinf 
\end{align*}
exist and are finite; note that  because of Lemma \ref{rem-orif} (see also \eqref{remw0}) they are real. Therefore Theorem~\ref{ess_spec1} applies to the transformed block operator matrix~$\cA_0$. 

The regular part of the essential spectrum can be read off immediately, using the relation
$\tDelta(x)=\Delta(-\log x)$, $x\in (0,1]$:
\[
  \sessreg(\wt\cA)=\sessreg(\cA) = \overline{\Delta([0,\infty))} = \ov{\,\tDelta((0,1])}\,.
\]
In order to determine the singular part of the essential spectrum, we note~that, by \eqref{pirhokappa}, 
\begin{align*}
   &\Re \dfrac{\trho(t,\la)}{\tpi(t,\la)} = - \Re \Big( x \frac{\ttrho(x,\la)}{\ttpi(x,\la)} \Big), \quad 
   \Im \dfrac{\trho(t,\la)}{\tpi(t,\la)} = - \Im \Big( x \frac{\ttrho(x,\la)}{\ttpi(x,\la)} \Big) + 1 + W(t), \\
   &\Re \dfrac{\tkappa(t,\la)}{\tpi(t,\la)} =  x^2 \frac{\tkappa(x,\la)}{\ttpi(x,\la)} + \frac 12 W(t) \Im \Big( x \frac{\ttrho(x,\la)}{\ttpi(x,\la)} \Big)
   -\! \frac 12 \Big( W(t) \!+\! \frac 12 W(t)^2 \!\!+\! W'(t) \Big)
\end{align*}
for $t\in[0,\infty)$, $x=\e^{-t}$. By Lemma~\ref{rem-orif0}, we have $\lim_{t\to\infty} W(t) = \Im \widetilde \rho_0(\la) -1$, $\lim_{t\to\infty}W'(t)=0$, and hence, by \eqref{mainass-t},
\begin{align*}
  & \Re \Big( \dfrac{\trho(\cdot,\la)}{\tpi(\cdot,\la)} \Big)\blinf =  - \Re \widetilde\rho_0(\lambda), \quad
  \Im \Big( \dfrac{\trho(\cdot,\la)}{\tpi(\cdot,\la)} \Big)\blinf =  0, \\
  & \Re \Big( \dfrac{\tka(\cdot,\la)}{\tpi(\cdot,\la)} \Big)\blinf = \Re \widetilde\kappa_0(\lambda) + \frac 12 (\Im \widetilde \rho_0(\la) -1)\Im \widetilde\rho_0(\lambda)- \frac 14\big( (\Im \widetilde\rho_0(\lambda) )^2 -1\big).
\end{align*}
Therefore the condition \eqref{def-discr} for a point $\la$ to belong to $\sesssing(\cA)$ takes the form
\begin{align*}
0 & \le \Big( \dfrac{\trho(\cdot,\la)}{\tpi(\cdot,\la)} \Big)\blinf^2\!\!-4 \Big( \dfrac{\tka(\cdot,\la)}{\tpi(\cdot,\la)} \Big)\blinf 
= \Re \Big( \dfrac{\trho(\cdot,\la)}{\tpi(\cdot,\la)} \Big)\blinf^2\!\!-4 \Re \Big( \dfrac{\tka(\cdot,\la)}{\tpi(\cdot,\la)} \Big)\blinf  \\
& = \big( \Re \widetilde\rho_0(\lambda) \big)^2 - 4 \Re \widetilde\kappa_0(\lambda)
- \big( \Im \widetilde\rho_0(\lambda) \big)^2  + 2 \Im \widetilde\rho_0(\lambda)  - 1\\
&=  \ds \, \Re \big( (\widetilde\rho_0(\lambda)-\ii)^2 \big) -  4\Re  \widetilde\kappa_0(\lambda).
\qedhere
\end{align*}
\end{proof}

\vspace{1mm}

\begin{remark}
Singular matrix differential operators on $(0,1]$ were also considered
in the papers \cite{MR2790893}, \cite{MR2793253}, but by a different method. 
There it is shown how to calculate the essential spectrum by means of a transformation to a
canonical system, but no explicit characterization of the essential
spectrum was given. Moreover, the assumptions used therein are not
comparable to ours, as the following examples~show.

If we consider an operator matrix of the form \eqref{matrix01} with $w_1\equiv w_2\equiv 1$ and  coefficient functions
\begin{alignat*}{2}
\tb(x)& = x\log\frac{x}{\e}, \quad & \tc(x)&=-\log\frac{x}{\e}, \quad \td(x)=-\log x,\\
\tp(x)& = -x^2\log\frac{x}{\e}, \quad &  \tq(x)&=1-2\log x, 
\end{alignat*}
for $x\in (0,1]$, then  our assumptions ($\tilde{\mbox{A}}$), ($\tilde{\mbox{B}}$), ($\tilde{\mbox{C}}$) are all satisfied and the essential spectrum is $\bigl[-1, -\frac{9}{13}\bigr]$, but \cite[Assumption
(H)]{MR2793253}  is violated since
\begin{equation*}
\ds  
\lim_{x\to 0+}\Big|\frac{\widetilde p(x)\widetilde c(x)}{\widetilde p(x)(\lambda\!-\!\widetilde d(x))\!+\!|\widetilde b(x)|^2}\Big|^2 \!\!=\!
\lim_{x\to 0+}\!\frac{(\log x\!-\!1)^2}{(\lambda+1)^2}\!=+\infty \ \ \textnormal{for every}\ \lambda\in\RR\setminus\{-1\}. 
\end{equation*}
On the other hand, for $\tb(x)=x(\log^2x-1)$, $\tc(x)=0$ for $x\in(0,1]$ and the other coefficient functions as above, \cite[Assumption
(H)]{MR2793253} is satisfied while ours are not since there does not exist a constant $\beta$ for which assumption ($\wt{\mbox{B}}2$)~holds.
\end{remark}

\section{Examples}
\label{sec5}

In this section we show how our method simplifies the calculation of the essential spectrum of singular matrix differential operators 
by means of three examples which were studied before using different methods. 

We begin with a problem which was studied in \cite{DG79} (see also \cite{MR580568}) and in \cite{MR2074778} as a model in the linear stability theory of plasmas confined to a toroidal region in~$\RR^3$. Eliminating one variable by means of the $S^1$-symmetry, one arrives at the following second order system of differential equations in the radial variable $x$ and the angular variable $\varphi$ on the cross section of the torus. 


\begin{example}
Let $\omega >0$ be a constant and $\Omega \defeq (0,1) \times (0,2\pi) \subset \RR^2$. Denote the variables
in $\Omega$ by $x, \varphi$ and the respective derivatives by $\partial_1, \partial_2$.
We introduce the operators $\mathfrak A_0, \mathfrak B_0, \mathfrak C_0$, and
$\mathfrak D_0$ in $L^2(\Omega, \frac 1x)$ by
\begin{align*}
  \mathfrak A_0 &:= - \frac 1{\omega^2} x \partial_1 \frac 1x \partial_1 - \frac 1{\omega^2} \partial_2^2, \qquad
  \mathfrak B_0 :=  - \frac 1\omega x \partial_1 \frac 1{x^2} \partial_2, \\
  \mathfrak C_0 &:= - \frac 1\omega \frac 1x \partial_2 \partial_1, \qquad \hspace*{16.5mm}
  \mathfrak D_0 := - \frac{1+x^2}{x^2}\partial_2^2
\end{align*}
with domains
\[
  \cD(\mathfrak A_0) = \cD(\mathfrak B_0) = \cD(\mathfrak C_0) = \cD(\mathfrak D_0) := C_{0,\pi}^\infty(\ov{\Omega})
\]
where
\begin{align*}
  C_{0,\pi}^\infty(\ov{\Omega})
  \defeq \{ f \in C^\infty(\ov{\Omega}) : {\rm supp} f \subset [\eps, 1-\eps] \times   [0,2\pi) \text{ for some } \eps \in (0,1/2),& \\
  \partial_2^j f(\cdot,0) = \partial_2^j f(\cdot,2 \pi), \,j\in\N_0 \}.\!\!\!&
\end{align*}
It was shown in \cite[Section~5]{MR2074778} that the operator $L_0$
in $L^2(\Omega, \frac 1x)^{2}$, 
given~by
\[
  L_0 \defeq \matrix{cc}{\mathfrak A_0 & \mathfrak B_0 \\ \mathfrak C_0 & \mathfrak D_0}, \quad \cD(L_0)
      := C_{0,\pi}^\infty(\ov{\Omega}) \oplus C_{0,\pi}^\infty(\ov{\Omega}),
\]
is symmetric and semi-bounded;  hence it possesses a Friedrichs extension $L_F$. Using Fourier series decomposition with respect
to the second variable $\varphi$, the operator $L_0$ becomes a direct~sum
\[
  L_0 = \!\!\!\!\sum_{m=-\infty}^\infty \!\!\!\! \wt \cA_{0,m}|_{C_0^\infty((0,1))^2}, \ \
  \wt \cA_{0,m} \!\!\defeq\!\! \matrix{cc}{\!\!- \ds x \dx   \frac 1{\omega^2} \frac 1x \dx + \frac {m^2}{\omega^2}
                           & \!- \ds \ii\frac m\omega x \dx \frac 1{x^2} \!\\[1.5ex] - \ds \ii\frac m\omega \frac 1x \dx & \ds m^2\frac{1+x^2}{x^2}}\!\!, \, m\in\Z.
\]
Here the block operator matrices $\wt \cA_{0,m}$ in the weighted Hilbert space product $L^2((0,1),\frac 1x)^{2}$, 
defined on $\cD(\wt \cA_{0,m})=C_0^2((0,1))\oplus C_0^1((0,1))$, have the form \eqref{matrix01} with
\begin{alignat*}{5}
\wt p(x) & = \frac 1{\omega^2} \frac 1x, \quad & \wt q(x) &= \frac {m^2}{\omega^2}, \quad & \wt b(x) &= - \ii \frac m\omega \frac 1x, \quad
& \wt c(x) &= 0, \quad & \wt d(x) &= m^2 \frac{1+x^2}{x^2}, \\
w_1(x) & = \frac 1x, & w_2(x)&=1, &  w(x)&=\frac 1x & & & &
\end{alignat*}
for $x\in(0,1]$; note that $C_0^\infty((0,1))^2$ is a core for $\wt\cA_{0,m}$ (compare Remark \ref{Cinfcore}). 
Moreover, $\wt \cA_{0,m}$ is semi-bounded and hence has a Friedrichs extension $\cA_{F,m}$.
The coefficients \eqref{tpirho}, \eqref{tka} of the Schur complement are given by
\begin{align*}
  \ttpi(x,\lambda) &= \frac {x^2}{\omega^2} \frac{m^2-\lambda}{m^2+x^2(m^2-\lambda)}, \quad 
  \ttrho(x,\lambda) = \mbox{i} \frac{\ttpi(x,\lambda)}{x} \frac{m^2-x^2(m^2-\lambda)}{m^2+x^2(m^2-\lambda)}, \\
  \tkappa(x,\lambda) &= \frac{m^2}{\omega^2} - \lambda
\end{align*}
for $x\in(0,1]$. It is easy to see  that the limits
\[
  \wt d_0 = + \infty, \quad \ttrho_0(\lambda) = \mbox{i}, \quad \tkappa_0(\lambda) = m^2 \frac{m^2 - \lambda \omega^2}{m^2-\lambda} 
\]
exist and all assumptions of Theorem \ref{essspec01} are satisfied. Now
\[
  \tDelta(x) = m^2\Bigl(1+\frac{1}{x^2}\Bigr)-\frac{1}{x} \omega^2x \frac{m^2}{\omega^2x^2}=m^2, \quad x\in(0,1],
\]
which shows that $\sessreg(\cA_{F,m}) = \{m^2\}$. 
Furthermore, $\la$ satisfies the condition $\Re \big( (\widetilde\rho_0(\lambda)-\ii)^2 \big) -  4\Re \big( \widetilde\kappa_0(\lambda)\big) \ge 0$ in Theorem \ref{essspec01} 
if and only if $\la$ lies between $m^2$ and~$\frac{m^2}{\omega^2}$. 
Hence Theorem \ref{essspec01} implies~that 
\begin{align*}
  \sess(\cA_{F,m}) &
                  =  \Big[ \min\Big\{ m^2, \frac{m^2}{\omega^2} \Big\}, \max\Big\{ m^2, \frac{m^2}{\omega^2} \Big\} \Big].
\end{align*}
So, if e.g.\ $0<\omega\le 1$, then the essential spectrum of $L_F$ exhibits a band structure,
\[
  \sess(L_F) = \bigcup_{m=-\infty}^\infty \Big[m^2, \frac{m^2}{\omega^2}\Big] = \bigcup_{m\in\N_0} \Big[m^2, \frac{m^2}{\omega^2}\Big].
\]
\end{example}

Next we consider an example whose essential spectrum was studied in a series of papers, see \cite{MR2388939} and the references therein, and which was also treated 
in \cite[Example~1 and Remark~4.2]{MR2793253} by transforming it to a canonical system.

\begin{example}
\label{ex7.2}
Let $\rho, m \in C^{2}([0,1],\RR)$, $\beta \in C^{2}([0,1],\CC)$ and $\phi \in C([0,1],\RR)$, and assume that $\rho(x) > 0$
for $x\in[0,1]$ and $m(0)\ne0$. We consider the block operator matrix
\[
  \wt\cA_0 = \begin{pmatrix} -\dx \rho \dx + \phi & \ds\dx \frac{\beta}{x} \\[2ex]
  \ds-\frac{\ov{\beta}}{x} \dx & \ds\frac{m}{x^2} \end{pmatrix}
\]
with domain $C_0^{2}((0,1)) \oplus C_0^{1}((0,1))$ 
in the Hilbert space $L^2(0,1)^{2}$. Here $w_1 = w_2 = w = 1$, $\td_0=+\infty$,~and 
\begin{align*}
  \ttpi(x,\lambda) &= \rho(x) - \frac{|\beta(x)|^2}{m(x)-\lambda x^2}, \quad \tkappa(x,\lambda) = \phi(x) - \lambda,\quad
  \ttrho(x,\lambda) = \mbox{i}\frac{\D}{\D x}\ttpi(x,\lambda)
\end{align*}
for $x\in(0,1]$.  By elementary arguments it can be shown that, under the above assumptions on the coefficients, the two conditions
\begin{equation}
\label{De0}
    (\rho m-|\beta|^2)(0)=0,  \quad  
    \lim_{x\to 0+}\frac {(\rho m - |\beta|^2)'(x)}{x} \ \mbox{ exists and is finite,}
\end{equation}
are equivalent to Assumptions ($\widetilde{\mbox{A}}$), ($\widetilde{\mbox{B}}$), and ($\widetilde{\mbox{C}}$). To see that they are sufficient, we use 
L'H\^opital's rule which yields the existence of the limits
\begin{align}
\label{quasi1}
  \lim_{x\to 0+}\frac {(\rho m - |\beta|^2)(x)}{x^2}, \quad 
  \lim_{x\to 0+} x\frac {(\rho m - |\beta|^2 - \la x^2 \rho)'(x)}{(\rho m - |\beta|^2 - \la x^2 \rho)(x)} = 2.
\end{align}
In particular, the first condition in \eqref{quasi1} implies that $\overline{\tDelta((0,1])}$ is bounded and
\begin{equation*}
\tDelta_0 \defeq \lim_{x\to 0+}\tDelta(x) = \lim_{x\to 0+} \frac 1{x^2} \Big( m(x) - \frac{|\beta(x)|^2}{\rho(x)} \Big) \ \mbox{ exists and is finite}.
\end{equation*}
Together with the smoothness assumptions on the coefficients and the conditions $\rho(0)>0$, $m(0) \ne 0$, it is easy to see that Assumptions $(\wt{\rm  B}2)$, 
$({\rm \wt B3a})$, and  $({\rm \wt B3b})$ are satisfied and that $({\rm \wt C1})$ holds with
\begin{align*}
    \ttrho_0(\lambda) &\!=\! \mbox{i}\!\lim_{x\to 0+}\!\!\frac{x\frac{\D}{\D x}\ttpi(x,\lambda)}{\ttpi(x,\lambda)}
    = \mbox{i}\!\lim_{x\to 0+}\! \!\Big( \!x\frac {(\rho m \!-\! |\beta|^2 \!-\! \la x^2 \rho)'(x)}{(\rho m \!-\! |\beta|^2 \!-\! \la x^2 \rho)(x)} 
    - x \frac{m'(x)\!-\!2\la x}{m(x)\!-\!\la x^2} \Big)
    \!=\!2\mbox{i},\\
    \tkappa_0(\lambda) &\!=\!\lim_{x\to 0+}\frac{(m(x)-\lambda x^2)(\phi(x)-\lambda)}{\rho(x)(\tDelta(x)-\lambda)}=\frac{m(0)(\phi(0)-\lambda)}{\rho(0)(\tDelta_0-\lambda)}.
\end{align*}
To see that the conditions \eqref{De0} are also necessary, we note that $({\rm \wt B3a})$, $({\rm \wt B3b})$, and $m(0)\ne 0$ imply that $(\rho m - |\beta|^2)(x) = { \rm O}(x^2)$,  $(\rho m - |\beta|^2)'(x) = { \rm O}(x)$ for $x\to 0+$. In particular, the first condition in \eqref{De0} holds and $\overline{\tDelta((0,1])}$ is bounded. Hence we can choose $\la \in \RR \setminus \overline{\tDelta((0,1])}$, $\la \ne \phi(0)$. Then the existence and finiteness of the limit $\kappa_0(\la)$ in $({\rm \wt C1})$ implies that the limit 
\[
 \lim_{x\to 0+} \frac{\ttpi(x,\la)}{x^2} = \frac{\phi(0)-\la}{\wt \kappa_0(\la)} \ne 0
\]
exists. Together with $({\rm \wt C1})$ it follows that the limit 
\[
 \lim_{x\to 0+} \ii  \frac{(\rho m - |\beta|^2 - \la x^2 \rho)'(x)}{x} = \ttrho_0(\lambda) \lim_{x\to 0+} \frac{\ttpi(x,\la)}{x^2} 
\]
exists and is finite, and hence the second condition in \eqref{De0} holds.

Altogether, if \eqref{De0} holds, we can apply Theorem \ref{essspec01} to calculate the singular part of the essential spectrum of 
any closed symmetric extension $\wt\cA$ of the  operator $\wt\cA_0$. Observing that
$\td_0=+\infty$, we obtain that, for $\la \notin\ov{\Delta([0,\infty))}$,
\begin{align*}
\la \in \sesssing(\wt\cA) 
\, \iff \, \Re \big( (\widetilde\rho_0(\lambda)\!-\!\ii)^2 \big) -  4\Re \big( \widetilde\kappa_0(\lambda)\big) \ge 0
\, \iff \, 1+4\tkappa_0(\lambda) \le 0;
\end{align*}
since $\rho(0)>0$, the first condition in \eqref{De0} implies that $m(0) >0$ and hence 
$\la \in \sesssing(\wt\cA)$ if and only if $\la$ lies between the two points
\[
  \tDelta_0 \quad \text{and} \quad \frac{4m(0)\phi(0)+\rho(0) \tDelta_0}{4m(0)+\rho(0)}. 
\]
Therefore the essential spectrum of every closed symmetric extension $\wt\cA$ of the  operator $\wt\cA_0$ in $L^2(0,1)^{2}$ is given~by
\begin{align*}
  \sess(\wt\cA) =& \ov{ \Big\{ \frac 1{x^2} \Big( m(x) - \frac{|\beta(x)|^2}{\rho(x)} \Big) : x\in (0,1] \Big\}} \\
              & \cup \!\Big[ \!\min \!\Big\{ \tDelta_0, \!\frac{4m(0)\phi(0)\!+\!\rho(0)\tDelta_0}{4m(0)\!+\!\rho(0)}\Big\}\!,
                          \max \Big\{ \tDelta_0, \!\frac{4m(0)\phi(0)\!+\!\rho(0)\tDelta_0}{4m(0)\!+\!\rho(0)}\Big\} \!\Big];
\end{align*}
note that, in fact, $\sess(\wt \cA)$ is just one interval since the end-point $\tDelta_0$ of the second interval lies in the first interval.
\end{example}

\begin{remark}
This result agrees with the results in \cite[Example~1, Case~3, and Remark~4.2]{MR2793253}, where it was already noted that
the description of the essential spectrum in \cite{MR2388939} and earlier papers 
is only valid if $\phi(0)=0$. 
Note that under the stronger assumptions $\phi \in C^2([0,1])$ imposed in \cite[(1.2)]{MR2388939}, the conditions 
\eqref{De0} are equivalent to the assumptions in \cite[(1.4)]{MR2388939}.
\end{remark}

Finally, the essential spectrum of the next example was studied  in \cite{MR2169702}.

\begin{example}
Let $\gamma, d_0 \in C^{2}([0,1],\RR)$, $\beta \in C^2([0,1],\CC)$,  $\gamma_1 \in C^1([0,1],\CC)$, and $\phi \in C([0,1],\RR)$
be such that $\gamma >0$ and $d_0 \ge 0$, $d_0(0)>0$.
We consider the block operator matrix
\[
  \wt \cA_0 \defeq \begin{pmatrix} \ds -\dx \frac{\gamma}{x} \dx x + \phi & \ds- \ii \dx \frac{\beta}{x} + \gamma_1 \\[2ex]
               - \ds \frac{\ov{\beta}}{x^2} \ii \dx x + \ov{\gamma_1} & \ds \frac{d_0}{x^2} \end{pmatrix}
\]
with domain $C_0^{2}((0,1)) \oplus C_0^{1}((0,1))$ 
in the product of weighted Hilbert spaces $L^2((0,1),x)^{2}$. Here we have $w_1(x)=1$, $w_2(x)=w(x)=x$ for $x\in[0,1]$, $c(x)= -\ii \frac{\ov{\beta(x)}}{x^2} + \ov{\gamma_1(x)}$ for $x\in (0,1]$,  $\td_0=+\infty$, and
\begin{align*}
  \ttpi(x,\lambda) =&\ \gamma(x)-\frac{|\beta(x)|^2}{d_0(x)-\lambda x^2}, \\
  \ttrho(x,\lambda)=&\ \frac{2x\Real(\beta(x)\ov{\gamma_1(x)})}{d_0(x)-\lambda x^2}+\mbox{i}\frac{\D}{\D x}\ttpi(x,\lambda)+\mbox{i}\frac{\ttpi(x,\lambda)}{x},  \\
  \tkappa(x,\lambda)=&\ \phi(x)-\lambda+\mbox{i}x\frac{\D}{\D x}\Bigl(\frac{\beta(x)\ov{\gamma_1(x)}}{d_0(x)-\lambda x^2}\Bigr)-\frac{x^2|\gamma_1(x)|^2}{d_0(x)-\lambda x^2}+\frac{\ttpi(x,\lambda)}{x^2}\\
                     &+\mbox{i}\frac{2\Real(\beta(x)\ov{\gamma_1(x)})}{d_0(x)-\lambda x^2}-\frac{1}{x}\frac{\D}{\D x}\ttpi(x,\lambda)
\end{align*}
for $x\in(0,1]$. In a similar way as in Example \ref{ex7.2}, one can show that the conditions 
\begin{equation}\label{De02}
(\gamma d_0 - |\beta|^2)(0)=0,  \quad  
    \lim_{x\to 0+}\frac {(\gamma d_0 - |\beta|^2)'(x)}{x} \ \mbox{ exists and is finite}    
\end{equation}
(compare \cite[(A2), (A3)]{MR2169702}) are equivalent to Assumptions (A), (B), and~(C).
Together with L'H\^opital's rule, we conclude that the limits 
\begin{align*}
& \tDelta_0 \defeq \lim_{x\to 0+}  \tDelta(x) = \lim_{x\to 0+} \frac 1{x^2} \Big( d_0(x) - \frac{|\beta(x)|^2}{\gamma(x)} \Big),
\\
&  \lim_{x\to 0+} \frac{\frac{\D}{\D x}\ttpi(x,\lambda)}{2x} =
\lim_{x\to 0} \frac{\ttpi(x,\lambda)}{x^2}=\lim_{x\to 0}\frac{\gamma(x)(\tDelta(x)-\lambda)}{d_0(x)-\lambda x^2}=\frac{\gamma(0)(\tDelta_0-\lambda)}{d_0(0)}
\end{align*}
exist and are finite and that 
\begin{align*}
\ttrho_0(\lambda) &=\frac{2\Re \bigl( \beta(0) \ov{\gamma_1(0)} \bigr) }{\gamma(0)(\tDelta_0-\lambda)}+3\mbox{i},\\
\tkappa_0(\lambda) &=\frac{d_0(0) \bigl(\phi(0)-\lambda\bigr) }{\gamma(0)(\tDelta_0-\lambda)}+\mbox{i}\frac{2\Real\bigl(\beta(0)\ov{\gamma_1(0)} \bigr) }{\gamma(0)(\tDelta_0-\lambda)}-1.
\end{align*}
Altogether, if \eqref{De0} holds, we can apply Theorem \ref{essspec01} to calculate the singular part of the essential spectrum of any closed symmetric extension $\wt \cA$ of the  operator $\wt \cA_0$. Observing that $\td_0=+\infty$, we have, for $\la\notin\ov{\Delta([0,\infty))}$,
\begin{align}
\nonumber
  \la \in \sesssing(\wt \cA) \ 
  &\iff \ \Re \big( (\widetilde\rho_0(\lambda)\!-\!\ii)^2 \big) -  4\Re \big( \widetilde\kappa_0(\lambda)\big) \ge 0 \\
\label{really-last}  
  & \iff \ \Big(\frac{\Re\bigl(\beta(0)\ov{\gamma_1(0)}\bigr)}{|\beta(0)|}\Big)^2\!
  - (\tDelta_0 - \la )\bigl(\phi(0)-\la\bigr) \ge 0;
\end{align}
here, to prove the equivalence \eqref{really-last}, we have multiplied the first inequality with $\frac{\gamma(0) (\tDelta_0-\la)^2}{4d_0(0)}$ ($>0$ for $\la\in \RR \setminus \sessreg(\cA)= \RR \setminus \overline{\tDelta((0,1])}$)
and used that $d_0(0) \gamma(0) \!=\! |\beta(0)|^2$ by \eqref{De02}. It is not difficult to see that \eqref{really-last} is equivalent \vspace{-2mm} to
\[
  \la \in [\la_-,\la_+], \quad
  \la_\pm \defeq \frac{\phi(0)+\tDelta_0}2 \pm \sqrt{ \Big( \frac{\phi(0)-\tDelta_0}2 \Big)^2
             + \Big(\frac{\Re \bigl( \beta(0) \ov{\gamma_1(0)}\bigr) }{|\beta(0)|}\Big)^2}.
\]
Hence we obtain that, for every closed symmetric extension $\wt \cA$ of the  operator $\wt \cA_0$ in $L^2((0,1),x)^2$,
\[
  \sess(\wt \cA) = \ov{ \Big\{ \frac 1{x^2} \Big( d_0(x) - \frac{|\beta(x)|^2}{\gamma(x)} \Big):
  x\in (0,1] \Big\} } \cup [\la_-,\la_+].
\]
\end{example}

\begin{remark}
This result coincides with \cite[Theorem~4.9]{MR2169702}. Since there semi-bounded operator matrices are considered and hence quadratic forms can be used, 
only the weaker assumptions $\gamma$, $d_0 \in C^1([0,1],\RR)$, $\beta \in C^1([0,1],\CC)$ are needed.
\end{remark}

\begin{remark}
Another example of a semi-bounded operator matrix \eqref{A0} on $\RR$ (with $p\equiv 1$) was considered in \cite{MR1885445}. 
Their assumptions are not comparable to ours; e.g.\ \cite[(4.2)]{MR1885445}
implies our assumption (B2), and \cite[(4.3)]{MR1885445} requires that $q$ has limit $0$ at $\infty$, while 
we do not suppose $q$ to have a limit at $\infty$ at all; on the other hand, we assume $d$ 
to have a (possibly improper) limit at $\infty$, while in \cite{MR1885445} the corresponding limit need not exist.
\end{remark}


\section{Application to a problem from astrophysics}
\label{sec6}

In this section we apply our main result on the essential spectrum to a spectral problem arising in the stability analysis of spherically symmetric stellar equilibrium models which was studied in \cite{MR1347113}.

For polytropic equilibrium models with constant adiabatic index near the centre and near the boundary of the star, the unperturbed part of the reduced spheroidal operator is a matrix differential operator in the radial variable $r$ having the form 
\[
  \wt\cA_0 = \begin{pmatrix} -\dr p_1 \dr + q_1 &  \ds \dr p_2+q_2 \\[2ex]
  \ds -p_2\dr+q_2                               &   p_3 \end{pmatrix}.
\]
We consider $\wt\cA_0$ in the Hilbert space $L^2(0,R)^{2}$ on the domain $C_0^{2}((0,R)) \oplus C_0^{1}((0,R))$, where $R>0$ is the radius of the star.
Using the notation of \cite{MR1347113}, 
the coefficient functions $p_1, p_2,$ $p_3, q_2 \in C^1((0,R], \RR)$, and $q_1 \in C((0,R], \RR)$ are given~by
\begin{alignat*}{2}
p_1:=& \frac{\Gamma_1 p}{\varrho}, \quad &q_1&:=\frac{1}{r\varrho}\cdot\bigl((4-3\Gamma_1)p\bigr)'+\frac{1}{r^2\sqrt{\varrho}}\cdot\Bigl(\frac{\Gamma_1p}{\varrho}\cdot (r^2\sqrt{\varrho})'\Bigr)', \\
p_2:=& c\frac{\Gamma_1p}{r\varrho}, \quad &q_2&:=c\frac{\Gamma_1p}{r\varrho}\cdot\Bigl(A-\frac{1}{2}\cdot\frac{(r^2\varrho)'}{r^2\varrho}\Bigr), \quad p_3:=c^2\frac{\Gamma_1p}{r^2\varrho}.
\end{alignat*}
Here the adiabatic index $\Gamma_1\!\in\! C^{2}([0,R],\RR)$ is positive on $[0,R]$ and assumed to be constant near the centre and the boundary of the star and thus $\Gamma_1(0), \Gamma_1(R) >0$ and $\Gamma_1'(0)=\Gamma_1'(R)=0$.
The  pressure 
$p\in C^2((0,R],\RR)$ and mass density $\varrho\in C^3((0,R],\RR)$ are positive, and 
$c\ge \sqrt 3$ is a constant. 
For a polytropic equilibrium model, $p$ and $\varrho$ are supposed to have the~form
\begin{align}
\label{prho}
p(r)=p_c\cdot(\theta_n(r))^{n+1},\quad \varrho(r)=\varrho_c\cdot(\theta_n(r))^n, \quad r\in (0,R),
\end{align}
where $p_c$, $\varrho_c >0$ are the constant central pressure and central density, respectively, of the unperturbed star.
The polytropic index $n$ is an element of the open interval $(0, 5)$,  and  $\theta=\theta_n \in C^2((0,R],\RR)$ 
is the Lane--Emden function of index $n$ which is uniformly positive and satisfies the non-linear differential equation 
\begin{align}
\label{L-M1}
\theta''_n(r)+\frac{2}{r}\theta'_n(r)=-\frac 1{\alpha_n^2} (\theta_n(r))^n, \quad r\in (0,\infty),
\end{align}  
where $\alpha_n\in\RR_+$ is the Lane--Emden unit length and $R=R_n >1$ is the first zero of $\theta=\theta_n$ (see e.g.\ \cite{MR0092663}). 
In particular, $\theta_n$ satisfies 
\begin{alignat}{2}
\label{L-M2}
& \ds \lim_{r\to 0+}\theta_n(r)=1, \quad && \lim_{r\to 0+}r^2\theta'_n(r)=0, \\
\label{L-M2R}
& \ds \ \theta_n > 0 \mbox{ on } [0,R), \quad && \lim_{r\to R-} \theta_n(r)=0. \qquad  
\end{alignat}
Note that, in the physically most interesting cases, it is assumed that 
$n\ge 1$ (see  \cite[Section 5, p.\ 47]{MR1347113}).
Finally, $A \in C^1((0,R),\RR)$ is connected with the buoyancy forces within the star and given as the component of $\rho^{-1} {\rm grad\,} \rho - (\Gamma_1 p)^{-1} {\rm grad\,} p$ in the radial direction.

From the above assumptions it follows that the coefficients are singular at $r=0$, and \vspace{-2mm} because of
\[
   \lim_{r\to R-} p_1(r) = \Gamma_1(R) \frac{p_c}{\varrho_c} \lim_{r\to R-} \theta_n(r) = 0
\]
also at $r=R$. Since $R>1$, Glazman's decomposition principle \eqref{glaz} yields that
\[
 \sess(\wt\cA) = \sess(\wt\cA_{(0,1)}) \cup \sess(\wt\cA_{(1,R)})
\]
where $\wt\cA_{(0,1)}$ and $\wt\cA_{(1,R)}$ are the minimal closed symmetric operators obtained when restricting $\wt\cA_0$ to the intervals $(0,1)$ and $(1,R)$, respectively.

In order to determine the regular part of the essential spectrum of arbitrary closed symmetric extensions $\wt\cA^0$ of $\wt\cA_{(0,1)}$ 
and $\wt\cA^R$ of $\wt\cA_{(1,R)}$,
we can use Theorem~\ref{essspec01}. To this end, we first note that $w_1 \equiv w_2 \equiv w \equiv 1$ and the function $\tDelta$ defined in \eqref{tDelta} is given by
\[
 \tDelta(r) = c^2\frac{\Gamma_1(r)p(r)}{r^2\varrho(r)}-\ds \Big( c\frac{\Gamma_1(r)p(r)}{r\varrho(r)} \Big)^2 
 \frac{\varrho(r)}{\Gamma_1(r) p(r)} = 0, \quad r\in(0,R).
\]
Hence, by Theorem \ref{essspec01},  the regular part of the essential spectrum is 
\[
  \sessreg(\wt\cA^0) = \sessreg(\wt\cA^R)  = \{0\}.
\]

In order to determine the singular part of the essential spectrum of $\wt\cA^0$, we note that, by \eqref{prho}, \eqref{L-M2} and since $\Gamma_1(0)>0$,  
\[
  d_0=\lim_{r\to 0+} p_3(r) = \lim_{r\to 0+} c^2\frac{\Gamma_1(r)p(r)}{r^2\varrho(r)} = 
  c^2 \frac{p_c}{\varrho_c} \Gamma_1(0) \lim_{r\to 0+}  \frac{\theta_n(r)}{r^2} = +\infty.
\]  
This shows that Assumption ($\widetilde{\mbox{B}}1$) holds. From the special form of the coefficients $p_2$, $p_3$ (corresponding to $b$, $d$, respectively, in Section \ref{sec4}), it is obvious that  ($\widetilde{\mbox{B}}2$) is satisfied.

Moreover, after some calculations one finds that the functions $\ttpi(\cdot,\la)$, $\ttrho(\cdot,\la)$, and $\tkappa(\cdot,\la)$ defined in
\eqref{tpirho}, \eqref{tka} are of the form 
\begin{align*}
\ttpi(r,\lambda)& = \frac{\lambda \Gamma_1(r) p(r)r^2}{\lambda r^2\varrho(r)-c^2\Gamma_1(r)p(r)},\qquad
\ttrho(r,\lambda) = \ii\ttpi(r,\lambda)\cdot \Bigl(\frac{2}{r}+\frac{\theta'_n(r)}{\theta_n(r)}+\Delta_1(r)\Bigr),\\
\tkappa(r,\lambda)& = \ttpi(r,\lambda)\Bigl(\frac{n^2}{4}\cdot\frac{\theta_n^{\prime\,2}(r)}{\theta_n^{2}(r)}+\frac{\Delta_2(r)}{r^2\theta^2_n(r)}+\bigl(1-\ii \frac{n+1}{\Gamma_1(r)}\bigr)\frac{c^2}{r^2}\Bigr).
\end{align*}
with certain functions $\Delta_1\in C([0,1],\CC)$, $\Delta_2\in C^1([0,1],\CC)$ such that 
\begin{equation}
\label{limit1}
  \lim_{r\to 0+} r\Delta_1(r) = 0, \quad 
  \lim_{r\to 0+} \Delta_2(r) = \Delta_2(0) =0
\end{equation}
(compare \cite[(4.2.3) and the following remark]{MR1347113}).
It is clear from (\ref{L-M1}) and the first relation in (\ref{L-M2}) that 
\begin{align}
\label{limit2}
\ds \lim_{r\to 0+}\bigl(r^2\theta''_n(r)+2r\theta'_n(r)\bigr)=-\frac{1}{\alpha^2_n}\lim_{r\to 0+}\bigl(r^2\theta_n(r)^n\bigr)=0.
\end{align}
Therefore L'H\^opital's rule and again the first relation in (\ref{L-M2}) imply that
\begin{align}
\label{limit3}
\ds 0=\lim_{r\to 0+}\frac{r^2\theta'_n(r)}{r}=\lim_{r\to 0+}r\theta'_n(r), 
\quad \lim_{r\to 0+}r \frac{\theta'_n(r)}{\theta_n(r)}=0.
\end{align}
It is not difficult to see that  \eqref{limit1}, \eqref{limit2}, \eqref{limit3} ensure that
the functions $\ttpi(r,\lambda)$, $\ttrho(r,\lambda)$, and $\tkappa(r,\lambda)$ above satisfy Assumptions ($\widetilde{\mbox{B}}3$a), 
($\widetilde{\mbox{B}}3$b), and ($\widetilde{\mbox{C}}1$) with
\begin{align*}
\ttrho_0(\lambda)=2\ii, \quad \tkappa_0(\lambda)=c^2\Bigl(1-\ii\frac{n+1}{\Gamma_1(0)}\Bigr).
\end{align*} 
Assumption ($\widetilde{\mbox{C}}2$) is trivially satisfied since $w_1=w_2=1$ and therefore $w''=(w_1w_2)''=0$.

Thus all assumptions of Theorem \ref{essspec01} are satisfied for the operator $\cA_{(0,1)}$. It follows that, 
for every closed symmetric extension $\cA^0$ of \vspace{-1mm} $\cA_{(0,1)}$, 
\[
\lambda\in \sesssing(\wt\cA^0) \ \iff \ 0 
\le  \Re \big( (\widetilde\rho_0(\lambda)-\ii)^2 \big) -  4\Re \big( \widetilde\kappa_0(\lambda)\big) 
= -1 -4 c^2,
\]
which is impossible, and hence 
\[ 
\sesssing(\wt\cA^0)=\emptyset, \quad \sess(\wt\cA^0) = \{0\}.
\] 

In order to determine the essential spectrum $\sesssing(\wt\cA^R)$ caused by the singularity at the boundary $R$ of the star, more work is needed. We conjecture that it will turn out to be empty as well. However, our present method is not readily applicable since the coefficients of the Schur complement are not bounded at $R$ because the first derivative of the Lane--Emden function $\theta=\theta_n$ does not vanish at $R$.

%

\section{Appendix}

Here we prove Proposition \ref{ha} on the adjoint of the matrix differential operator $\cA_0$ defined in \eqref{A0}. To this end we first prove a lemma which may be of independent interest. 

\begin{lemma}\label{adjoint}
Let $T$ be a densely defined symmetric operator in a Hilbert space $\cH$ with scalar product $(\cdot,\cdot)$,
let $g_1$, \dots, $g_n \in\cH$ and set
\[
  \cD_0 \defeq \{u\in\cD(T)\colon  (u,g_j)=0, j=1,\dots,n\}.
\]
If $x\in\cH$ is such that 
\begin{equation}\label{501}
  \cD_0 \to \CC, \quad u \mapsto (Tu,x) \quad \mbox{is bounded},
\end{equation}
then $x\in\cD(T^*)$.
\end{lemma}

\begin{proof} 
We have to show that the mapping in \eqref{501} is bounded on $\cD(T)$ or, equivalently, continuous in $0$. 

Without loss of generality we may assume that $g_i\ne 0$, $i\in \{1,\dots,n\}$; zero elements $g_j$ may be omitted, and if all $g_j$ are zero there is nothing to prove.
First we show that there exist linearly independent $f_1,\dots, f_n \in \cD(T)$ with
$(f_i,g_i) \ne 0$, $i=1,\dots,n$. Indeed, there are $\wt f_1, \dots, \wt f_n \in \cH$ with $(\wt f_i,g_j) = \delta_{ij}$, $i,j=1,\dots,n$.
Since $\cD(T)$ is dense in $\cH$, we can choose 
$f_1,\dots, f_n \in \cD(T)$ such that $\|\wt f_i - f_i \| < 1 / (n \, \max_{j=1,\dots,n} \|g_j\| )$, $i=1,\dots,n$.
Then $|(f_i,g_i)| > 1 - 1/n$ and $|(f_i,g_j)| < 1/n$, $i,j=1,\dots,n$, $i\ne j$. Hence the matrix $\phi_n:=\big( (f_i,g_j) \big)_{i,j=1}^n$ is strictly diagonally dominant, thus invertible, and so $f_1,\dots, f_n$ are linearly independent.

Now let $(u_k)_{k=1}^\infty \subset \cD(T)$ with $u_k\to 0$, $k\to\infty$. Since  
$\cD(T)$ can be decomposed as $\cD(T) = \cD_0 \,\dot{+}\, \spn\{f_1,\dots,f_n\}$, 
there exist sequences $(u_k^0)_{k=1}^\infty \subset \cD_0$ and $(\alpha_k^1)_{k=1}^\infty, \dots, (\alpha_k^n)_{k=1}^\infty \subset \CC$
with 
\[
 u_k = u_k^0 + \alpha_k^1 f_1 + \dots + \alpha_k^n f_n, \quad k\in \N.
\]
Taking the scalar product with $g_1,\dots,g_n$ and noting that the matrix $\phi_n=\big( (f_i,g_j) \big)_{i,j=1}^n$ is invertible (see above), we find that
\[
 \big( \alpha_k^j \big)_{j=1}^n = \phi_n^{-1} \big( (u_k, g_j) \big)_{j=1}^n \to 0, \quad k\to \infty,
\]
since $u_k\to 0$, $k\to\infty$.  Then, if $x\in \cH$, we have
\[
  (Tu_k, x) = (T u_k^0, x) + \alpha_k^1 (T f_1, x) + \dots + \alpha_k^n (Tf_n,x), \quad k\in \N.
\]
The first term on the right hand side tends to $0$ by assumption since $u_k^0 \in \cD_0$ and all other terms tend to $0$ since $\alpha_k^j \to 0$, $k\to\infty$, for $j=1,\dots,n$. Thus  $(Tu_k, x) \to 0$, $k\to \infty$, as required.
\end{proof}

\begin{proof}[Proof of Proposition~{\rm \ref{ha}}]
The symmetry of $\cA_0$ is easy to check, and so is the inclusion ``$\supset$'' in \eqref{domA0adj}: If $y$ is in the set on the right-hand side of \eqref{domA0adj} and $f\in\cD(\cA_0)$, then it is easily seen that $(\cA_0 f,y)=(f,w)$
with $w$ equal to the right-hand side of \eqref{A0adj}; hence $y\in\cD(\cA_0^*)$.

In order to show the inclusion ``$\subset$'' in \eqref{domA0adj} and the equality \eqref{A0adj},
let \linebreak $y=\binom{y_1}{y_2}\in\cD(\cA_0^*)$
with $w = \binom{w_1}{w_2} \defeq \cA_0^*y \,\in L^2(0,\infty)^2$.  Then
\[
  (\cA_0f,y) = (f,\cA_0^*y), \quad f=\binom{f_1}{f_2}\in\cD(\cA_0)= C_0^2((0,\infty))\oplus C_0^1((0,\infty)),
\]
or, equivalently,
\begin{equation}\label{505}
\begin{aligned}
  &\bigl(-(pf_1')'\!+\!qf_1\!-\!(\ov{b} f_2)'\!+\!\ov{c}f_2,\,y_1\bigr) 
  \!+\! \bigl(bf_1'\!+\!cf_1+df_2,\,y_2\bigr)
  \!=\! (f_1,w_1) \!+\! (f_2,w_2) \hspace{-6mm}
\end{aligned}
\end{equation}
for all $f_1 \in C_0^2((0,\infty))$, $f_2 \in C_0^1((0,\infty))$.
In the following, let $t_0>0$ be arbitrary  and consider $f_1$, $f_2$ with compact support contained in $(0,t_0)$, i.e.\ $f_1\in C_0^2((0,t_0))$, $f_2 \in C_0^1((0,t_0))$.
In particular, when $f_2=0$, \eqref{505} yields that
\begin{equation}\label{503}
  \bigl(-(pf_1')'+qf_1,y_1\bigr) + \bigl(bf_1'+cf_1,y_2\bigr) = (f_1,w_1), \quad f_1 \in C_0^{2}((0,t_0)).
\end{equation}
Let $T$ be the operator in $L^2(0,t_0)$ defined by
\begin{align*}
  \cD(T) &:= C_0^1((0,t_0)), \quad   T v :=  v',
\end{align*}
which is injective with inverse $T^{-1}$ bounded on $\ran T = \spn \{1\}^\perp$. 
By Assumption (A), we have $p \in C^1([0,t_0])$, $\frac 1p \in C^1([0,t_0]) \subset L^2(0,t_0)$ since $p>0$ on $[0,\infty)$. Thus
\[
 \cD_0:= \{ pf_1': f_1 \in C_0^2((0,t_0))\} = \Big\{ h \in C^1_0((0,t_0)) = \cD( \ii T): h \perp \frac 1p \Big\} \subset \ran (pT)
\] 
and the operator $p T$ has an inverse $(pT)^{-1}=T^{-1} \frac 1p$ which is bounded on~$\ran (pT)$.
Assumption (A) also ensures that the multiplication operators by $q$, $\frac 1p$, $b$, and $c$ are bounded in $L^2(0,t_0)$.
Thus, from \eqref{503} we conclude that the mapping
\begin{equation}\label{504}
  \varphi \mapsto \bigl(\varphi',y_1\bigr) = \bigl(q (p T)^{-1} \varphi,y_1\bigr)
  + \Big( \tfrac bp \varphi+c(pT)^{-1} \varphi,y_2\Big) - \bigl((pT)^{-1} \varphi,w_1\bigr)
\end{equation}
is bounded on $\cD_0$. Now Lemma \ref{adjoint} applied to the symmetric operator $\ii T$ with $g=\frac 1p \in L^2(0,t_0)$ yields that $y_1 \in \cD((\ii T)^*)=\{v\in L^2(0,t_0): v,v'\in {\rm{AC_{loc}}}([0,t_0))\}$; in particular, $y_1$ is absolutely continuous on $(0,t_0)$.

This allows us to conclude from \eqref{505} that
\begin{equation}\label{507}
\begin{aligned}
  \bigl(pf_1'\!\!+\!\ov{b}f_2,y_1'\bigr) \!+\! \bigl(qf_1\!+\!\ov{c}f_2,y_1\bigr)
  \!+\! \bigl(bf_1'\!\!+\!cf_1\!+\!df_2,y_2\bigr) \!=\! \bigl(f_1,w_1\bigr) \!+\! \bigl(f_2,w_2\bigr)
\end{aligned}
\end{equation}
for all $f_1 \in C_0^2((0,t_0))$, $f_2 \in C_0^1((0,t_0))$.
In particular, again setting $f_2=0$, we find that
\[
  \bigl(f_1',\,py_1'+\ov{b}y_2\bigr) = - \bigl(f_1,\,qy_1+\ov{c}y_2\bigr) + \bigl(f_1,w_1\bigr), \quad f_1\in C_0^{2}((0,t_0)).
\]
Since $C_0^{2}((0,t_0))$ is a core for $T$ and the right hand side is bounded in $f_1$, it follows that 
$py_1'+\ov{b}y_2 \in \cD(T^*)$; in particular,  $py_1'+\ov{b}y_2$ is absolutely continuous on $(0,t_0)$, and 
\[
  -\bigl(f_1,(py_1'+\ov{b}y_2)'\bigr) + \bigl(f_1,\,qy_1+\ov{c}y_2\bigr)
  = \bigl(f_1,w_1\bigr), \quad f_1\in C_0^{2}((0,t_0)).
\]
Because $C_0^{2}((0,t_0))$ is dense in $L^2(0,t_0)$, this shows that
\[
  w_1 = -(py_1'+\ov{b}y_2)' + qy_1 + \ov{c}y_2 \quad \mbox{ on } (0,t_0).
\]
If we set $f_1=0$ in \eqref{507}, we obtain
\[
  (f_2,by_1') + (f_2,cy_1) + (f_2,dy_2) = (f_2,w_2), \quad f_2\in C_0^{1}((0,t_0)),
\]
which implies that
\[
  w_2 = by_1' + cy_1 + dy_2 \quad \mbox{ on } (0,t_0).
\]
Since $t_0>0$ was arbitrary and $w_1,w_2\in L^2(0,\infty)$ by assumption, it follows that
$y$ belongs to the set on the right-hand side of \eqref{domA0adj} and $\cA_0^*y$ is
equal to the right-hand side of \eqref{A0adj}.

It remains to be proved that  the deficiency indices of $\cA_0$ are $\le 2$.
Let $\lambda\in\CC\setminus\RR$ and $y=\binom{y_1}{y_2}\in\ker(\cA_0^*-\lambda)$, i.e.\
$\binom{y_1}{y_2}\in\cD(\cA_0^*)$ such that
\begin{align}
\label{i}
  -\bigl(py_1'+\ov{b}y_2\bigr)'+qy_1+\ov{c}y_2 &= \lambda y_1, \\
\label{ii}
  by_1'+cy_1+dy_2 &= \lambda y_2.
\end{align}
Solving \eqref{ii} for $y_2$ 
and multiplying by $\ov{b}$, we conclude that  
\begin{equation}
\label{last}
  \ov{b} y_2 = -  \frac{|b|^2}{d-\la} y_1' - \frac{\ov{b}c}{d-\la} y_1
\end{equation}
and hence 
\[
  \pi(\cdot,\la) y_1'= \left( p - \frac{|b|^2}{d-\la} \right) y_1' = ( py_1'+\ov{b} y_2 ) + \frac{\ov{b}c}{d-\la} y_1.
\]
This shows that $\pi(\cdot,\la) y_1' \in {\rm AC_{loc}}([0,\infty))$ since  $y_1$, $py_1'+\ov{b} y_2 \in {\rm AC_{loc}}([0,\infty))$ because 
$\binom{y_1}{y_2}\in\cD(\cA_0^*)$ and $\frac{\ov{b}c}{d-\la} \in C^1([0,\infty))$ because $\la\in\CC\setminus\RR$. 
Plugging \eqref{last} into \eqref{i}, we obtain that $y_1$ is a solution of the second order differential equation $S(\lambda)y_1=0$ where $S(\la)$ is the Schur complement given by \eqref{S}. 
By \cite[Theorem~III.10.1]{MR89b:47001} (see Remark \ref{schurdef}) there are at most two linearly independent functions $y_1$ that satisfy $S(\lambda)y_1=0$.
Since $y_2$ is uniquely determined by $y_1$, it follows that $\dim\ker(\cA_0^*-\lambda)\le2$.  
\end{proof}

{\small

\bibliographystyle{plain}
\bibliography{glasgowref}

}

\end{document}